\documentclass{article}
\usepackage{amsfonts}
\usepackage{amssymb}
\usepackage{amsmath}

\newtheorem{theorem}{Theorem}

\newenvironment{proof}[1][Proof]{\noindent\textbf{#1.} }{\ \rule{0.5em}{0.5em}}
\input{tcilatex}

\begin{document}

\title{A Lagrangian Method for Deriving New Indefinite Integrals of Special
Functions}
\author{John T. Conway \\
University of Agder, Grimstad, Norway\\
Tel: (47) 37 23 32 67 Fax: (47)\ 3723 30 01\\
Email: John.Conway@uia.no}
\maketitle
\date{}

\begin{abstract}
A new method is presented for obtaining indefinite integrals of common
special functions. The approach is based on a Lagrangian formulation of the
general homogeneous linear ordinary differential equation of second order. A
general integral is derived which involves an arbitrary function, and
therefore yields an infinite number of indefinite integrals for any special
function which obeys such a differential equation. Techniques are presented
to obtain the more interesting integrals generated by such an approach, and
many integrals, both previously known and completely new are derived using
the method. Sample results are given for Bessel functions, Airy functions,
Legendre functions and hypergeometric functions. More extensive results are
given for the complete elliptic integrals of the first and second kinds.
Integrals can be derived which combine common special functions as separate
factors.
\end{abstract}

\section{Introduction}

In tables of integrals such as [1-4], the bulk of the content consists of
definite integrals, with almost token collections of indefinite integrals.
However, the relatively few indefinite integrals which have been published
tend to be well known and heavily used. This paper presents a new and
surprisingly simple method for deriving indefinite integrals of both
elementary and special functions, provided the function satisfies an
ordinary linear differential equation of the form%
\begin{equation}
y^{\prime \prime }\left( x\right) +p\left( x\right) y^{\prime }\left(
x\right) +q\left( x\right) y\left( x\right) =0\text{.}  \label{Eqn0}
\end{equation}%
The principal result derived here is the indefinite integral%
\begin{equation*}
\dint f\left( x\right) \left( h^{\prime \prime }\left( x\right) +h^{\prime
}\left( x\right) p\left( x\right) +h\left( x\right) q\left( x\right) \right)
y\left( x\right) dx=
\end{equation*}%
\begin{equation}
f\left( x\right) \left( h^{\prime }\left( x\right) y\left( x\right) -h\left(
x\right) y^{\prime }\left( x\right) \right)  \label{Eqn1}
\end{equation}%
where $y\left( x\right) $ is any solution of equation (\ref{Eqn0}), $f\left(
x\right) $ is given by%
\begin{equation}
f\left( x\right) =\exp \left( \dint p\left( x\right) dx\right)  \label{Eqn1a}
\end{equation}%
and $h\left( x\right) $ is an arbitrary function. As $h\left( x\right) $ is
arbitrary, this equation yields an infinite number of integrals for any
function $y\left( x\right) $ which satisfies equation (\ref{Eqn0}). Many
well-known integrals in the literature can be derived very simply from
equation (\ref{Eqn1}), together with a large number of interesting new
integrals. Techniques for exploiting equation (\ref{Eqn1}) to obtain
interesting integrals are presented here, together with sample results for
selected special functions. Somewhat more than sample results are presented
for the complete elliptic integrals of the first and second kinds, but it is
impossible to do more than scratch the surface for any particular function
in a single paper. It seems that the total number of indefinite integrals
for special functions can be increased by a large factor using equation (\ref%
{Eqn1}). In particular, it is possible to derive integrals combining
products of different special functions, such as Bessel and Legendre
functions combined with each other, or with elliptic integrals, or other
special functions. Equation (\ref{Eqn1}) was originally derived from an
Euler-Lagrange equation, but once known, it can be proved in an elementary
manner without variational calculus and both derivations are given here.

Section 2 below presents the Euler-Lagrange formulation of equation (\ref%
{Eqn0}) and the derivation of equation (\ref{Eqn1}) from it, together with a
simpler proof. A second integral is also derived, which will not be
considered in detail here. Section 3 presents techniques for exploiting
equation (\ref{Eqn1}) and some well-known results are derived using this
method. It is shown that functions which are conjugate, in the sense that
the differential equations they satisfy have the same $p\left( x\right) $ in
equation (\ref{Eqn0}), can always be combined as products in the same
integral. It is shown that any two equations can be transformed to be
mutually conjugate and transformation equations are given to calculate $%
q\left( x\right) $ for arbitrarily specified $p\left( x\right) $. It is
shown that the reverse process, finding a suitable $p\left( x\right) $ to
give a desired $q\left( x\right) $, is governed by a Riccati equation [5].
Both known and new integrals are derived in this section as examples of the
methods presented. Section 4 presents sample results for a selection of
special functions. Section 5 gives somewhat more than sample results for the
complete elliptic integrals of the first and second kinds, using the method
of fragmentary equations presented in Section 3. Table 1 gives the special
functions used. 
\begin{table}[tbph]
\begin{equation*}
\begin{tabular}{ll}
Symbol & Special Function \\ 
$\func{Ai}\left( x\right) $ & Airy function of the first kind \\ 
$\func{Bi}\left( x\right) $ & Airy function of the second kind \\ 
$\func{Ai}^{\prime }\left( x\right) ,\func{Bi}^{\prime }\left( x\right) $ & 
Derivatives of the Airy functions \\ 
$E(k)$ & Complete elliptic integral of the second kind \\ 
${}_{2}F_{1}\left( a,b;c;x\right) $ & Gauss hypergeometric function \\ 
$\func{Gi}\left( x\right) ,\func{Hi}\left( x\right) $ & Scorer functions \\ 
$I_{n}\left( x\right) $ & Modified Bessel function of the first kind \\ 
$J_{\nu }(x)$ & Bessel function of the first kind \\ 
$H_{n}\left( x\right) $ & Struve function of the first kind \\ 
$K(k)$ & Complete elliptic integral of the first kind \\ 
$K_{n}\left( x\right) $ & Modified Bessel function of the second kind \\ 
$P_{\nu }\left( x\right) $ & Legendre function of the first kind \\ 
$P_{\nu }^{\mu }\left( x\right) $ & Associated Legendre function of the
first kind \\ 
$Q_{\nu }\left( x\right) $ & Legendre function of the second kind \\ 
$Q_{\nu }^{\mu }\left( x\right) $ & Associated Legendre function of the
second kind \\ 
$s_{m,n}\left( x\right) $ & Lommel function \\ 
$Y_{\nu }\left( x\right) $ & Bessel function of the second kind \\ 
$\Gamma \left( x\right) $ & Gamma function%
\end{tabular}%
\end{equation*}%
\caption{Special Functions Used}
\end{table}

\section{Formulation}

Equation (\ref{Eqn0}) can be expressed in Lagrangian form as:%
\begin{equation}
\frac{d}{dx}\left( \frac{\partial \tciLaplace }{\partial y^{\prime }}\right)
-\frac{\partial \tciLaplace }{\partial y}=0  \label{Eqn2}
\end{equation}%
where the Lagrangian $\tciLaplace $ is given by%
\begin{equation}
\tciLaplace =f\left( x\right) \left( y^{\prime 2}\left( x\right) -q\left(
x\right) y^{2}\left( x\right) \right)  \label{Eqn3}
\end{equation}%
and the function $f\left( x\right) $ is a solution of the equation: 
\begin{equation}
\frac{f^{\prime }\left( x\right) }{f\left( x\right) }=p\left( x\right)
\label{Eqn4}
\end{equation}%
and hence is given by the integral (\ref{Eqn1a}). A principal property of
any Lagrangian is:%
\begin{equation}
\frac{d}{dx}\left( \tciLaplace -y^{\prime }\frac{\partial \tciLaplace }{%
\partial y^{\prime }}\right) =\frac{\partial \tciLaplace }{\partial x}
\label{Eqn6}
\end{equation}%
and from equation (\ref{Eqn3})\ we obtain%
\begin{equation}
\tciLaplace -y^{\prime }\frac{\partial \tciLaplace }{\partial y^{\prime }}%
=-f\left( x\right) \left( y^{\prime 2}\left( x\right) +q\left( x\right)
y^{2}\left( x\right) \right)  \label{Eqn7}
\end{equation}%
and%
\begin{equation}
\frac{\partial \tciLaplace }{\partial x}=f\left( x\right) ^{\prime
}y^{\prime 2}\left( x\right) -\left( f\left( x\right) q\left( x\right)
\right) ^{\prime }y^{2}\left( x\right) \text{.}  \label{Eqn8}
\end{equation}%
Substituting equations (\ref{Eqn7})\ and (\ref{Eqn8}) into equation (\ref%
{Eqn6}) gives immediately the integral%
\begin{equation}
\dint \left( \left( f\left( x\right) q\left( x\right) \right) ^{\prime
}y^{2}\left( x\right) -f^{\prime }\left( x\right) y^{\prime 2}\left(
x\right) \right) dx=f\left( x\right) \left( y^{\prime 2}\left( x\right)
+q\left( x\right) y\left( x\right) \right)  \label{Eqn9}
\end{equation}%
which is analogous to the energy integral in classical mechanics. This
integral can be considered to be complementary to equation (\ref{Eqn1}), as
it gives different integrals, but these results will be presented
separately. Explicitly evaluating both derivatives in equation (\ref{Eqn6})
using equations (\ref{Eqn7})\ and (\ref{Eqn8}) gives after cancellation and
collection of terms:%
\begin{equation}
-\frac{d}{dx}\left( f\left( x\right) y^{\prime }\left( x\right) \right)
=f\left( x\right) q\left( x\right) y\left( x\right)  \label{Eqn10}
\end{equation}%
and we obtain a second indefinite integral:%
\begin{equation}
\dint f\left( x\right) q\left( x\right) y\left( x\right) dx=-f\left(
x\right) y^{\prime }\left( x\right) \text{.}  \label{Eqn11}
\end{equation}%
The differential equation (\ref{Eqn1}) can be transformed by defining $%
y\left( x\right) =h\left( x\right) z\left( x\right) $ to give:%
\begin{equation}
z^{\prime \prime }+\left( 2\frac{h^{\prime }\left( x\right) }{h\left(
x\right) }+p\left( x\right) \right) z^{\prime }+\left( \frac{h^{\prime
\prime }\left( x\right) }{h\left( \left( x\right) \right) }+\frac{h^{\prime
}\left( x\right) }{h\left( x\right) }p\left( x\right) +q\left( x\right)
\right) z=0\text{.}  \label{Eqn12}
\end{equation}%
This transformed equation can also be put in the Lagrangian form%
\begin{equation}
\frac{d}{dx}\left( \frac{\partial \bar{\tciLaplace}}{\partial z^{\prime }}%
\right) -\frac{\partial \bar{\tciLaplace}}{\partial z}=0  \label{Eqn13}
\end{equation}%
where%
\begin{equation}
\bar{\tciLaplace}=\bar{f}\left( x\right) \left( z^{\prime 2}\left( x\right) -%
\bar{q}\left( x\right) z\left( x\right) \right)  \label{Eqn14}
\end{equation}%
and 
\begin{equation}
\bar{q}\left( x\right) =\frac{h^{\prime \prime }\left( x\right) }{h\left(
x\right) }+\frac{h^{\prime }\left( x\right) }{h\left( x\right) }p\left(
x\right) +q\left( x\right) \text{.}  \label{Eqn15}
\end{equation}%
The function $\bar{f}\left( x\right) $, the equivalent of $f\left( x\right) $
in equation (\ref{Eqn3}), is given by 
\begin{equation}
\bar{f}\left( x\right) =\exp \left( \dint \left( 2\frac{h^{\prime }\left(
x\right) }{h\left( x\right) }+p\left( x\right) \right) dx\right)
=h^{2}\left( x\right) f\left( x\right)  \label{Eqn16}
\end{equation}%
so that the new Lagrangian can be expressed as%
\begin{equation}
\bar{\tciLaplace}=h^{2}\left( x\right) f\left( x\right) \left( z^{\prime
2}-\left( \frac{h^{\prime \prime }\left( x\right) }{h\left( x\right) }+\frac{%
h^{\prime }\left( x\right) }{h\left( x\right) }p\left( x\right) +q\left(
x\right) \right) z^{2}\left( x\right) \right) \text{.}  \label{Eqn17}
\end{equation}%
The equivalent of the integral (\ref{Eqn11}) for the transformed equation is%
\begin{equation}
\dint h^{2}\left( x\right) f\left( x\right) \left( \frac{h^{\prime \prime
}\left( x\right) }{h\left( x\right) }+\frac{h^{\prime }\left( x\right) }{%
h\left( x\right) }p\left( x\right) +q\left( x\right) \right) z\left(
x\right) dx=-h^{2}\left( x\right) f\left( x\right) z^{\prime }\left( x\right)
\label{Eqn18}
\end{equation}%
and since $z\left( x\right) =y\left( x\right) /h\left( x\right) $ and hence 
\begin{equation}
z^{\prime }\left( x\right) =\frac{y^{\prime }\left( x\right) }{h\left(
x\right) }-\frac{h^{\prime }\left( x\right) y\left( x\right) }{h^{2}\left(
x\right) }  \label{Eqn19}
\end{equation}%
then equation (\ref{Eqn18} ) reduces to equation (\ref{Eqn1}). This relation
holds for all homogeneous second-order linear ordinary differential
equations for an arbitrary function $h\left( x\right) $. Once known, this
result can be proven very simply without reference to variational calculus.

\begin{theorem}
\begin{equation*}
\dint f\left( x\right) \left( h^{\prime \prime }\left( x\right) +h^{\prime
}\left( x\right) p\left( x\right) +h\left( x\right) q\left( x\right) \right)
y\left( x\right) dx=
\end{equation*}%
\begin{equation}
f\left( x\right) \left( h^{\prime }\left( x\right) y\left( x\right) -h\left(
x\right) y^{\prime }\left( x\right) \right)  \label{Eqn20}
\end{equation}%
where $f\left( x\right) $ and $y\left( x\right) $ obey the respective
differential equations:%
\begin{equation}
f^{\prime }\left( x\right) =p\left( x\right) f\left( x\right)  \label{Eqn21}
\end{equation}%
\begin{equation}
y^{\prime \prime }\left( x\right) +p\left( x\right) y^{\prime }\left(
x\right) +q\left( x\right) y\left( x\right) =0  \label{Eqn22}
\end{equation}%
and $h\left( x\right) $, $p\left( x\right) $ and $q\left( x\right) $ are
arbitrary functions.
\end{theorem}

\begin{proof}
Arbitrary functions $f\left( x\right) $, $h\left( x\right) $ and $y\left(
x\right) $ satisfy the differential identity: 
\begin{equation}
\frac{d}{dx}\left( f\left( x\right) \left( h^{\prime }\left( x\right)
y\left( x\right) -h\left( x\right) y^{\prime }\left( x\right) \right)
\right) =f^{\prime }\left( x\right) \left( h^{\prime }\left( x\right)
y\left( x\right) -h\left( x\right) y^{\prime }\left( x\right) \right)
\label{Eqn23}
\end{equation}%
\begin{equation*}
+f\left( x\right) \left( h^{\prime \prime }\left( x\right) y\left( x\right)
-h\left( x\right) y^{\prime \prime }\left( x\right) \right) \text{.}
\end{equation*}%
Eliminating $f^{\prime }\left( x\right) $ and $y^{\prime \prime }\left(
x\right) $ from equation (\ref{Eqn23}) using equations (\ref{Eqn21})\ and (%
\ref{Eqn22}), respectively, gives%
\begin{equation*}
\frac{d}{dx}\left( f\left( x\right) \left( h^{\prime }\left( x\right)
y\left( x\right) -h\left( x\right) y^{\prime }\left( x\right) \right)
\right) =
\end{equation*}%
\begin{equation}
f\left( x\right) \left( h^{\prime \prime }\left( x\right) +h^{\prime }\left(
x\right) p\left( x\right) +h\left( x\right) q\left( x\right) \right) y\left(
x\right)  \label{Eqn24}
\end{equation}%
for arbitrary $h\left( x\right) $ and integration of equation (\ref{Eqn24})
gives (\ref{Eqn20}) and the theorem is proven. Equation (\ref{Eqn21}) can be
integrated immediately to give%
\begin{equation}
f\left( x\right) =\exp \left( \dint p\left( x\right) dx\right) \text{.}
\label{Eqn25}
\end{equation}
\end{proof}

In equation (\ref{Eqn1}) we can take $y\left( x\right) \equiv y_{1}\left(
x\right) $ to be any solution of equation (\ref{Eqn0}) and $h\left( x\right)
=y_{2}\left( x\right) $ to be another solution, to obtain the Wronskian for
the differential equation (\ref{Eqn0}) as: 
\begin{equation}
W\left( x\right) \equiv 
\begin{vmatrix}
y_{1}\left( x\right) & y_{2}\left( x\right) \\ 
y_{1}^{\prime }\left( x\right) & y_{2}^{\prime }\left( x\right)%
\end{vmatrix}%
=\frac{A}{f\left( x\right) }  \label{Eqn26}
\end{equation}%
where $A$ is a constant. From equation (\ref{Eqn1a}) this is equivalent to
Abel's formula [6,7] for the Wronskian:%
\begin{equation}
W\left( x\right) =W\left( x_{0}\right) \exp \left(
-\dint\limits_{x_{0}}^{x}p\left( x\right) dx\right) \text{.}  \label{Eqn27}
\end{equation}

\section{Techniques for finding indefinite integrals}

Equation (\ref{Eqn1}) yields an infinite number of indefinite integrals for
each solution $y\left( x\right) $ of equation (\ref{Eqn1}), as any $h\left(
x\right) $ yields an integral. In equation (\ref{Eqn1}) we can take $h\left(
x\right) =1$ or substitute any elementary function such as $e^{ax}$ or $%
e^{-ax^{2}}$ and obtain an integral for any special function. Expressions
involving elementary functions, such as $h\left( x\right) =x^{m}$ and $%
h\left( x\right) =x^{m}\QATOPD\{ \} {\sin x}{\cos x}$ often yield integrals
which are interesting and new. We can also specify $h\left( x\right) $ to be
the solution to a differential equation where one or two terms in equation (%
\ref{Eqn1}) is deleted, such as:%
\begin{equation}
h^{\prime \prime }\left( x\right) +p\left( x\right) h^{\prime }\left(
x\right) =0  \label{Eqn28}
\end{equation}%
\begin{equation}
p\left( x\right) h^{\prime }\left( x\right) +q\left( x\right) h\left(
x\right) =0  \label{Eqn29}
\end{equation}%
\begin{equation}
h^{\prime \prime }\left( x\right) +q\left( x\right) h\left( x\right) =0\text{%
.}  \label{Eqn30}
\end{equation}%
It is frequently the case that $p\left( x\right) $ or $q\left( x\right) $
may consist of several terms, and $h\left( x\right) $ can be specified as
the solution of a differential equation where some of these terms have been
deleted. These equations with one or more terms deleted will be referred to
as fragmentary equations.

\subsection{Integration of the function $y\left( x\right) $ itself.}

To integrate a solution $y\left( x\right) $ of equation (\ref{Eqn1}) it is
necessary to specify $h\left( x\right) $ to be a solution of the
inhomogeneous equation:%
\begin{equation}
h^{\prime \prime }\left( x\right) +p\left( x\right) h^{\prime }\left(
x\right) +q\left( x\right) h\left( x\right) =\frac{1}{f\left( x\right) }.
\label{Eqn31}
\end{equation}%
For the Legendre equation:%
\begin{equation}
y^{\prime \prime }\left( x\right) -\frac{2x}{1-x^{2}}y^{\prime }\left(
x\right) +\frac{\nu \left( \nu +1\right) }{1-x^{2}}y\left( x\right) =0
\label{Eqn32}
\end{equation}%
then%
\begin{equation}
f\left( x\right) \equiv \exp \left( -\dint \frac{2x}{1-x^{2}}dx\right)
=1-x^{2}  \label{Eqn33}
\end{equation}%
and equation (\ref{Eqn31}) becomes%
\begin{equation}
h^{\prime \prime }\left( x\right) -\frac{2x}{1-x^{2}}h^{\prime }\left(
x\right) +\frac{\nu \left( \nu +1\right) }{1-x^{2}}h\left( x\right) =\frac{1%
}{1-x^{2}}  \label{Eqn34}
\end{equation}%
which has a simple solution%
\begin{equation}
h\left( x\right) =\frac{1}{\nu \left( \nu +1\right) }\text{.}  \label{Eqn35}
\end{equation}%
Substituting (\ref{Eqn35}) into equation (\ref{Eqn1}) for Legendre's
equation gives%
\begin{equation}
\dint \QATOPD\{ \} {P_{\nu }\left( x\right) }{Q_{\nu }\left( x\right) }dx=%
\frac{x^{2}-1}{\nu \left( \nu +1\right) }\QATOPD\{ \} {P_{\nu }^{\prime
}\left( x\right) }{Q_{\nu }^{\prime }\left( x\right) }  \label{Eqn36}
\end{equation}%
and the Legendre recurrence relations [1]:%
\begin{equation}
\left( x^{2}-1\right) \QATOPD\{ \} {P_{\nu }^{\prime }\left( x\right)
}{Q_{\nu }^{\prime }\left( x\right) }=\left( \nu +1\right) \QATOPD\{ \}
{P_{\nu +1}\left( x\right) -xP_{\nu }\left( x\right) }{Q_{\nu +1}\left(
x\right) -xQ_{\nu }\left( x\right) }  \label{Eqn37}
\end{equation}%
\begin{equation}
\left( x^{2}-1\right) \QATOPD\{ \} {P_{\nu }^{\prime }\left( x\right)
}{Q_{\nu }^{\prime }\left( x\right) }=\nu \QATOPD\{ \} {xP_{\nu }\left(
x\right) -P_{\nu -1}\left( x\right) }{xQ_{\nu }\left( x\right) -Q_{\nu
-1}\left( x\right) }  \label{Eqn38}
\end{equation}%
then give the alternative forms%
\begin{equation}
\dint \QATOPD\{ \} {P_{\nu }\left( x\right) }{Q_{\nu }\left( x\right) }dx=%
\frac{1}{\nu }\QATOPD\{ \} {P_{\nu +1}\left( x\right) -xP_{\nu }\left(
x\right) }{Q_{\nu +1}\left( x\right) -xQ_{\nu }\left( x\right) }
\label{Eqn39}
\end{equation}%
\begin{equation}
\dint \QATOPD\{ \} {P_{\nu }\left( x\right) }{Q_{\nu }\left( x\right) }dx=%
\frac{1}{\nu +1}\QATOPD\{ \} {xP_{\nu }\left( x\right) -P_{\nu -1}\left(
x\right) }{xQ_{\nu }\left( x\right) -Q_{\nu -1}\left( x\right) }\text{.}
\label{Eqn40}
\end{equation}%
The Legendre equation above is a simple case, and sometimes the solution to
equation (\ref{Eqn31}) is given in terms of functions specifically defined
to satisfy this inhomogeneous equation . Examples are the Lommel [8,9,] and
Struve [9] functions for Bessel's equation and the Scorer functions for
Airy's equation [2,10]. The Bessel equation is%
\begin{equation}
y^{\prime \prime }\left( x\right) +\frac{1}{x}y^{\prime }\left( x\right)
+\left( 1-\frac{n^{2}}{x^{2}}\right) y\left( x\right) =0  \label{Eqn41}
\end{equation}%
and has $f\left( x\right) $ given by%
\begin{equation}
f\left( x\right) \equiv \exp \left( \dint \frac{dx}{x}\right) =x\text{.}
\label{Eqn42}
\end{equation}%
The Lommel function $s_{m,n}\left( x\right) $ satisfies the equation 
\begin{equation}
y^{\prime \prime }\left( x\right) +\frac{1}{x}y^{\prime }\left( x\right)
+\left( 1-\frac{n^{2}}{x^{2}}\right) y\left( x\right) =x^{m-1}  \label{Eqn43}
\end{equation}%
with equation (\ref{Eqn31}) being the special case for $m=0$. Taking $%
h\left( x\right) =s_{m,n}\left( x\right) $ in equation (\ref{Eqn1}) gives%
\begin{equation}
\dint x^{m}J_{n}\left( x\right) dx=x\left( s_{m,n}^{\prime }\left( x\right)
J_{n}\left( x\right) -s_{m,n}\left( x\right) J_{n}^{\prime }\left( x\right)
\right)  \label{Eqn44}
\end{equation}%
which from the Lommel and Bessel recurrence relations [1]:%
\begin{equation}
s_{m,n}^{\prime }\left( x\right) =\left( m+n-1\right) s_{m-1,n-1}\left(
x\right) J_{n}\left( x\right) -\frac{n}{x}s_{m,n}\left( x\right) J_{n}\left(
x\right)  \label{Eqn45}
\end{equation}%
\begin{equation}
\frac{n}{x}J_{n}\left( x\right) +J_{n}^{\prime }\left( x\right)
=J_{n-1}\left( x\right)  \label{Eqn46}
\end{equation}%
reduces to%
\begin{equation}
x\left( s_{m,n}^{\prime }\left( x\right) J_{n}\left( x\right) -s_{m,n}\left(
x\right) J_{n}^{\prime }\left( x\right) \right)  \label{Eqn47}
\end{equation}%
and gives%
\begin{equation}
\dint x^{m}J_{n}\left( x\right) dx=x\left( \left( m+n-1\right)
s_{m-1,n-1}\left( x\right) J_{n}\left( x\right) -s_{m,n}\left( x\right)
J_{n-1}\left( x\right) \right) \text{.}  \label{Eqn48}
\end{equation}%
Watson [9] gives a rather different derivation of this result. For the
special case where $m=n$, substituting the identities [1]%
\begin{equation}
s_{n,n}\left( x\right) =\Gamma \left( n+1/2\right) \sqrt{\pi }2^{n-1}\mathbf{%
H}_{n}\left( x\right)  \label{Eqn49}
\end{equation}%
\begin{equation}
\left( n-\frac{1}{2}\right) \Gamma \left( n-\frac{1}{2}\right) =\Gamma
\left( n+\frac{1}{2}\right)  \label{Eqn50}
\end{equation}%
into equation (\ref{Eqn48}) gives the integral in terms of the Struve
function $\mathbf{H}_{n}\left( x\right) $ as: 
\begin{equation}
\dint x^{n}J_{n}\left( x\right) dx=\sqrt{\pi }2^{n-1}\Gamma \left(
n+1/2\right) x\left( \mathbf{H}_{n-1}\left( x\right) J_{n}\left( x\right) -%
\mathbf{H}_{n}\left( x\right) J_{n-1}\left( x\right) \right) \text{.}
\label{Eqn51}
\end{equation}%
The Airy functions are a similar case to that of the Bessel functions
described above. The Airy functions $\func{Ai}\left( x\right) $ and $\func{Bi%
}\left( x\right) $ obey the differential equation:%
\begin{equation}
y^{\prime \prime }\left( x\right) -xy\left( x\right) =0  \label{Eqn51a}
\end{equation}%
and the Scorer functions $\func{Gi}\left( x\right) $ and $\func{Hi}\left(
x\right) $ obey the inhomogeneous Airy equations [10]:%
\begin{equation}
y^{\prime \prime }\left( x\right) -xy\left( x\right) =\mp \frac{1}{\pi }
\label{Eqn51b}
\end{equation}%
where the negative sign is taken for $\func{Gi}\left( x\right) $ and the
positive sign for $\func{Hi}\left( x\right) $. Taking $h\left( x\right) $ in
equation (\ref{Eqn1}) to be either $\func{Gi}\left( x\right) $ or $\func{Hi}%
\left( x\right) $ gives the alternative integrals, which of course differ
only by a constant:%
\begin{equation}
\dint \QATOPD\{ \} {\func{Ai}\left( x\right) }{\func{Bi}\left( x\right)
}dx=\pi \QATOPD\{ \} {\func{Ai}^{\prime }\left( x\right) \func{Gi}\left(
x\right) -\func{Ai}\left( x\right) \func{Gi}^{\prime }\left( x\right) }{%
\func{Bi}^{\prime }\left( x\right) \func{Gi}\left( x\right) -\func{Bi}\left(
x\right) \func{Gi}^{\prime }\left( x\right) }  \label{Eqn51c}
\end{equation}%
\begin{equation}
\dint \QATOPD\{ \} {\func{Ai}\left( x\right) }{\func{Bi}\left( x\right)
}dx=\pi \QATOPD\{ \} {\func{Ai}\left( x\right) \func{Hi}^{\prime }\left(
x\right) -\func{Ai}^{\prime }\left( x\right) \func{Hi}\left( x\right) }{%
\func{Bi}\left( x\right) \func{Hi}^{\prime }\left( x\right) -\func{Bi}%
^{\prime }\left( x\right) \func{Hi}\left( x\right) }\text{.}  \label{Eqn51d}
\end{equation}

\subsection{Conjugate differential equations}

Two differential equations of the form (\ref{Eqn0}) are considered conjugate
if they have the same $p\left( x\right) $, and hence the same $f\left(
x\right) $, but different $q\left( x\right) $. For example, the equation 
\begin{equation}
y^{\prime \prime }\left( x\right) +\frac{1}{x}y^{\prime }\left( x\right) +%
\frac{1}{1-x^{2}}y\left( x\right) =0  \label{Eqn52}
\end{equation}%
is conjugate to the Bessel equation (\ref{Eqn41}) and has as a general
solution [1] the complete elliptic integral expression%
\begin{equation}
y\left( x\right) =C_{1}\mathbf{E}\left( x\right) +C_{2}\left( \mathbf{E}%
\left( x^{\prime }\right) -\mathbf{K}\left( x^{\prime }\right) \right)
\label{Eqn53}
\end{equation}%
where $x^{\prime }\equiv \sqrt{1-x^{2}}$ is the complementary modulus. The
Bessel equation (\ref{Eqn41})\ can be slightly generalized as 
\begin{equation}
y^{\prime \prime }\left( x\right) +\frac{1}{x}y^{\prime }\left( x\right)
+\left( \alpha ^{2}-\frac{n^{2}}{x^{2}}\right) y\left( x\right) =0
\label{Eqn54}
\end{equation}%
which has the solution $y\left( x\right) =C_{1}J_{n}\left( \alpha x\right)
+C_{2}Y_{n}\left( \alpha x\right) $ and two equations of the form (\ref%
{Eqn54}) are mutually conjugate unless both $\alpha $ and $n$ are the same
for both. The modified Bessel equation 
\begin{equation}
y^{\prime \prime }\left( x\right) +\frac{1}{x}y^{\prime }\left( x\right)
-\left( \alpha ^{2}+\frac{n^{2}}{x^{2}}\right) y\left( x\right) =0
\label{Eqn55}
\end{equation}%
with solution $y\left( x\right) \allowbreak =C_{1}I_{n}\left( \alpha
x\right) +C_{2}K_{n}\left( \alpha x\right) $ is always conjugate to equation
(\ref{Eqn54}).

In equation (\ref{Eqn1}) we can take $h\left( x\right) $ to be any solution
of a conjugate equation with $\bar{q}\left( x\right) $ instead of $q\left(
x\right) $. We have%
\begin{equation}
h^{\prime \prime }\left( x\right) +p\left( x\right) h^{\prime }\left(
x\right) =-\bar{q}\left( x\right) h  \label{Eqn56}
\end{equation}%
and hence 
\begin{equation}
h^{\prime \prime }\left( x\right) +p\left( x\right) h^{\prime }\left(
x\right) +q\left( x\right) h\left( x\right) =\left( q\left( x\right) -\bar{q}%
\left( x\right) \right) h\left( x\right)  \label{Eqn57}
\end{equation}%
so that equation (\ref{Eqn1}) gives the integral%
\begin{equation}
\dint f\left( x\right) \left( q\left( x\right) -\bar{q}\left( x\right)
\right) h\left( x\right) y\left( x\right) dx=f\left( x\right) \left(
h^{\prime }\left( x\right) y\left( x\right) -h\left( x\right) y^{\prime
}\left( x\right) \right)  \label{Eqn58}
\end{equation}%
If in this equation we take $y\left( x\right) =$ $Z_{n}\left( \alpha
x\right) $, where $Z_{n}\left( \alpha x\right) \equiv aJ_{n}\left( \alpha
x\right) +bY_{n}\left( x\right) $ is any solution of the Bessel equation (%
\ref{Eqn54}), and $h\left( x\right) =$ $Z_{m}\left( \beta x\right) $ is any
solution of the Bessel Equation

\begin{equation}
h^{\prime \prime }\left( x\right) +\frac{1}{x}h^{\prime }\left( x\right)
+\left( \beta ^{2}-\frac{m^{2}}{x^{2}}\right) h\left( x\right) =0
\label{Eqn59}
\end{equation}%
then equation (\ref{Eqn58}) gives the well-known Bessel integral [1]:%
\begin{equation*}
\dint \left( \left( \alpha ^{2}-\beta ^{2}\right) x-\frac{n^{2}-m^{2}}{x}%
\right) Z_{n}\left( \alpha x\right) Z_{m}\left( \beta x\right) dx=
\end{equation*}%
\begin{equation}
x\left( \beta Z_{m}^{\prime }\left( \beta x\right) Z_{n}\left( \alpha
x\right) -Z_{m}\left( \beta x\right) Z_{n}^{\prime }\left( \alpha x\right)
\right) \text{.}  \label{Eqn60}
\end{equation}%
Similarly, taking $y\left( x\right) $ to be any solution of equation (\ref%
{Eqn54}) and $h\left( x\right) =\bar{Z}_{m}\left( \beta x\right) $, where $%
\bar{Z}_{m}\left( \beta x\right) \equiv AI_{n}\left( \beta x\right)
+BK_{n}\left( \beta x\right) $ is any solution of the modified Bessel
equation:%
\begin{equation}
h^{\prime \prime }\left( x\right) +\frac{1}{x}h^{\prime }\left( x\right)
+\left( \beta ^{2}-\frac{m^{2}}{x^{2}}\right) h\left( x\right) =0
\label{Eqn61}
\end{equation}%
gives the integral%
\begin{equation*}
\dint \left( \left( \alpha ^{2}+\beta ^{2}\right) x-\frac{n^{2}-m^{2}}{x}%
\right) Z_{n}\left( \alpha x\right) \bar{Z}_{m}\left( \beta x\right) dx=
\end{equation*}%
\begin{equation}
x\left( \beta \bar{Z}_{m}^{\prime }\left( \beta x\right) Z_{n}\left( \alpha
x\right) -\bar{Z}_{m}\left( \beta x\right) Z_{n}^{\prime }\left( \alpha
x\right) \right) \text{.}  \label{Eqn62}
\end{equation}%
In equation (\ref{Eqn58}) we can take $y\left( x\right) =Z_{n}\left( \alpha
x\right) $ and $h\left( x\right) =\mathbf{E}\left( x\right) $, a solution of
the conjugate equation (\ref{Eqn52}), to obtain%
\begin{equation*}
\dint x\left( \alpha ^{2}-\frac{n^{2}}{x^{2}}-\frac{1}{x^{\prime 2}}\right)
Z_{n}\left( \alpha x\right) \mathbf{E}\left( x\right) dx=
\end{equation*}%
\begin{equation}
x\left( \frac{d\mathbf{E}\left( x\right) }{dx}Z_{n}\left( \alpha x\right) -%
\mathbf{E}\left( x\right) \frac{dZ_{n}\left( \alpha x\right) }{dx}\right)
\label{Eqn63}
\end{equation}%
where [1]:%
\begin{equation}
\frac{d\mathbf{E}\left( x\right) }{dx}=\frac{\mathbf{E}\left( x\right) -%
\mathbf{K}\left( x\right) }{x}\text{.}  \label{Eqn64}
\end{equation}%
A particularly simple special case of equation (\ref{Eqn64}) is%
\begin{equation}
\dint \frac{x^{3}}{x^{\prime 2}}J_{0}\left( x\right) \mathbf{E}\left(
x\right) dx=J_{0}\left( x\right) \left( \mathbf{K}\left( x\right) -\mathbf{E}%
\left( x\right) \right) -xJ_{1}\left( x\right) \mathbf{E}\left( x\right) 
\text{.}  \label{Eqn65}
\end{equation}

\subsection{Transformations of the differential equations}

Any differential equation of the form (\ref{Eqn0}) can be made conjugate to
any other by a change of dependent variable. This allows integrals to be
obtained containing a product of any two specified special functions,
provided they both satisfy differential equations of the form (\ref{Eqn0}).
The differential equation 
\begin{equation}
z^{\prime \prime }\left( x\right) +\bar{p}\left( x\right) z^{\prime }\left(
x\right) +\bar{q}\left( x\right) z\left( x\right) =0  \label{Eqn66}
\end{equation}%
can be made conjugate to equation (\ref{Eqn0}) by the transformation $%
y\left( x\right) =g\left( x\right) z\left( x\right) $ which yields the
equation%
\begin{equation}
y^{\prime \prime }\left( x\right) +\left( \frac{2g^{\prime }\left( x\right) 
}{g\left( x\right) }+\bar{p}\left( x\right) \right) y^{\prime }\left(
x\right) +\left( \frac{g^{\prime \prime }\left( x\right) }{g\left( x\right) }%
+\frac{g^{\prime }\left( x\right) }{g\left( x\right) }p\left( x\right) +\bar{%
q}\left( x\right) \right) y\left( x\right) =0  \label{Eqn67}
\end{equation}%
and we can choose 
\begin{equation}
\frac{2g^{\prime }\left( x\right) }{g\left( x\right) }+\bar{p}\left(
x\right) =p\left( x\right)  \label{Eqn68}
\end{equation}%
so that 
\begin{equation}
g\left( x\right) =\exp \left( \frac{1}{2}\dint \left( p\left( x\right) -\bar{%
p}\left( x\right) \right) dx\right) \text{.}  \label{Eqn69}
\end{equation}%
From equation (\ref{Eqn69}) we obtain the relation:%
\begin{equation}
g\left( x\right) =\sqrt{\frac{f\left( x\right) }{\bar{f}\left( x\right) }%
\text{.}}  \label{Eqn70}
\end{equation}%
The coefficient of $y\left( x\right) $ in equation (\ref{Eqn67}) can be
evaluated directly from $g\left( x\right) $ but it is usually more
convenient to note that%
\begin{equation}
\frac{d}{dx}\left( \frac{g^{\prime }\left( x\right) }{g\left( x\right) }%
\right) =\frac{g^{\prime \prime }\left( x\right) }{g\left( x\right) }-\left( 
\frac{g^{\prime }\left( x\right) }{g\left( x\right) }\right) ^{2}
\label{Eqn71}
\end{equation}%
and 
\begin{equation}
\frac{g^{\prime }\left( x\right) }{g\left( x\right) }=\frac{1}{2}\left(
p\left( x\right) -\bar{p}\left( x\right) \right)  \label{Eqn72}
\end{equation}%
so that equation (\ref{Eqn67}) becomes%
\begin{equation}
y^{\prime \prime }\left( x\right) +p\left( x\right) y^{\prime }\left(
x\right) +\left( \frac{1}{2}\left( p\left( x\right) -\bar{p}\left( x\right)
\right) ^{\prime }+\frac{p^{2}\left( x\right) -\bar{p}^{2}\left( x\right) }{4%
}+\bar{q}\left( x\right) \right) y\left( x\right) =0\text{.}  \label{Eqn73}
\end{equation}%
For example, the equation satisfied by the complete elliptic integral of the
first kind $\mathbf{K}\left( x\right) $ is%
\begin{equation}
z^{\prime \prime }\left( x\right) +\left( \frac{1}{x}-\frac{2x}{1-x^{2}}%
\right) z^{\prime }\left( x\right) -\frac{1}{1-x^{2}}z\left( x\right) =0
\label{Eqn74}
\end{equation}%
which has $\bar{f}\left( x\right) =x\left( 1-x^{2}\right) $ and general
solution [1]:%
\begin{equation}
z\left( x\right) =C_{1}\mathbf{K}\left( x\right) +C_{2}\mathbf{K}\left(
x^{\prime }\right) \text{.}  \label{Eqn75}
\end{equation}%
The transformed equation conjugate with the Bessel equation is given by
equation (\ref{Eqn73}) as%
\begin{equation}
y^{\prime \prime }\left( x\right) +\frac{1}{x}y^{\prime }\left( x\right) +%
\frac{1}{\left( 1-x^{2}\right) ^{2}}y\left( x\right) =0  \label{Eqn76}
\end{equation}%
and from the relation (\ref{Eqn70}) the general solution of equation (\ref%
{Eqn76}) is given by%
\begin{equation}
y\left( x\right) =C_{1}x^{\prime }\mathbf{K}\left( x\right) +C_{2}x^{\prime }%
\mathbf{K}\left( x^{\prime }\right) \text{.}  \label{Eqn77}
\end{equation}%
In equation (\ref{Eqn58}) we can take $y\left( x\right) $ to be $Z_{n}\left(
\alpha x\right) $, the general solution of the Bessel equation (\ref{Eqn54})
and specify $h\left( x\right) =x^{\prime }\mathbf{K}\left( x\right) $. This
gives the integral%
\begin{equation*}
\dint xx^{\prime }\left( \alpha ^{2}-\frac{n^{2}}{x^{2}}-\frac{1}{x^{\prime
4}}\right) Z_{n}\left( \alpha x\right) \mathbf{K}\left( x\right) dx=
\end{equation*}%
\begin{equation}
x\left( Z_{n}\left( \alpha x\right) \frac{d}{dx}\left( x^{\prime }\mathbf{K}%
\left( x\right) \right) -x^{\prime }\mathbf{K}\left( x\right) \frac{%
dZ_{n}\left( \alpha x\right) }{dx}\right)  \label{Eqn78}
\end{equation}%
where [1]:%
\begin{equation}
\frac{d\mathbf{K}\left( x\right) }{dx}=\frac{\mathbf{E}\left( x\right) }{%
xx^{\prime 2}}-\frac{\mathbf{K}\left( x\right) }{x}\text{.}  \label{Eqn79}
\end{equation}%
A special case of equation (\ref{Eqn78}) is%
\begin{equation}
\dint \frac{x^{3}\left( 2-x^{2}\right) }{x^{\prime 3}}\QATOPD\{ \}
{J_{0}\left( x\right) }{Y_{0}\left( x\right) }\mathbf{K}\left( x\right)
dx=\QATOPD\{ \} {J_{0}\left( x\right) }{Y_{0}\left( x\right) }\frac{\mathbf{K%
}\left( x\right) -\mathbf{E}\left( x\right) }{x^{\prime }}-xx^{\prime
}\QATOPD\{ \} {J_{1}\left( x\right) }{Y_{1}\left( x\right) }\mathbf{K}\left(
x\right) \text{.}  \label{Eqn80}
\end{equation}%
Employing the conjugate equations (\ref{Eqn52}) and (\ref{Eqn76}) in
equation (\ref{Eqn58}), with $y\left( x\right) =\mathbf{E}\left( x\right) $
and $h\left( x\right) =x^{\prime }\mathbf{K}\left( x\right) $ gives the
integral 
\begin{equation}
\dint \left( \frac{x}{x^{\prime }}\right) ^{3}\mathbf{E}\left( x\right) 
\mathbf{K}\left( x\right) dx=\left( \mathbf{K}\left( x\right) -\mathbf{E}%
\left( x\right) \right) \left( \frac{\mathbf{E}\left( x\right) }{x^{\prime }}%
-x^{\prime }\mathbf{K}\left( x\right) \right) \text{.}  \label{Eqn81}
\end{equation}%
If the common $p\left( x\right) $ of two conjugate differential equations is
transformed such that $p\left( x\right) \rightarrow \hat{p}\left( x\right) $
and $f\left( x\right) \rightarrow \bar{f}\left( x\right) $ for both
equations, the integral given by equation (\ref{Eqn58}) remains unchanged.
This can be shown by applying the transformations $\bar{y}\left( x\right)
=g\left( x\right) y\left( x\right) $ and $\bar{h}\left( x\right) =g\left(
x\right) h\left( x\right) $ to the conjugate equations: 
\begin{equation}
y^{\prime \prime }\left( x\right) +p\left( x\right) y^{\prime }\left(
x\right) +q\left( x\right) y\left( x\right) =0  \label{Eqn82}
\end{equation}%
\begin{equation}
h^{\prime \prime }\left( x\right) +p\left( x\right) h^{\prime }\left(
x\right) +\bar{q}\left( x\right) h\left( x\right) =0  \label{Eqn83}
\end{equation}%
to obtain the still conjugate differential equations: 
\begin{equation}
y^{\prime \prime }\left( x\right) +\hat{p}\left( x\right) y^{\prime }\left(
x\right) +q_{1}\left( x\right) y\left( x\right) =0  \label{Eqn84}
\end{equation}%
\begin{equation}
\bar{h}^{\prime \prime }\left( x\right) +\hat{p}\left( x\right) \bar{h}%
^{\prime }\left( x\right) +\bar{q}_{1}\left( x\right) \bar{h}\left( x\right)
=0  \label{Eqn85}
\end{equation}%
where from equation (\ref{Eqn73})%
\begin{equation}
q_{1}\left( x\right) =\frac{1}{2}\left( \hat{p}\left( x\right) -p\left(
x\right) \right) ^{\prime }+\frac{\hat{p}^{2}\left( x\right) -p^{2}\left(
x\right) }{4}+q\left( x\right)  \label{Eqn86}
\end{equation}%
\begin{equation}
\bar{q}_{1}\left( x\right) =\frac{1}{2}\left( \hat{p}\left( x\right)
-p\left( x\right) \right) ^{\prime }+\frac{\hat{p}^{2}\left( x\right)
-p^{2}\left( x\right) }{4}+\bar{q}\left( x\right)  \label{Eqn87}
\end{equation}%
and hence%
\begin{equation}
q_{1}\left( x\right) -\bar{q}_{1}\left( x\right) =q\left( x\right) -\bar{q}%
\left( x\right) \text{.}  \label{Eqn88}
\end{equation}%
From equation (\ref{Eqn70}) the general solutions of the differential
equations (\ref{Eqn84}) and (\ref{Eqn85}) are related to the general
solutions of equations (\ref{Eqn82}) and (\ref{Eqn83}), respectively, by:%
\begin{equation}
\bar{y}\left( x\right) =\sqrt{\frac{f\left( x\right) }{\bar{f}\left(
x\right) }}y\left( x\right)  \label{Eqn89}
\end{equation}%
\begin{equation}
\bar{h}\left( x\right) =\sqrt{\frac{f\left( x\right) }{\bar{f}\left(
x\right) }}h\left( x\right)  \label{Eqn90}
\end{equation}%
Substituting equations (\ref{Eqn88})-(\ref{Eqn90}) into equation (\ref{Eqn58}%
) yields after cancellation%
\begin{equation}
\dint f\left( x\right) \left( q\left( x\right) -\bar{q}\left( x\right)
\right) h\left( x\right) y\left( x\right) dx=\bar{f}\left( x\right) \left( 
\bar{h}^{\prime }\left( x\right) \bar{y}\left( x\right) -\bar{h}\left(
x\right) \bar{y}^{\prime }\left( x\right) \right)  \label{Eqn91}
\end{equation}%
and hence the transformed equations do not yield a new integral in this
manner, though new integrals can be obtained from the transformed equations
using fragmentary equations.

\subsection{A Riccati Equation}

Instead of transforming $p\left( x\right) $ in equation (\ref{Eqn0}), we can
instead target $q\left( x\right) $ for simplification by a dependent
variable change. The differential equation%
\begin{equation}
z^{\prime \prime }\left( x\right) +\bar{p}\left( x\right) z^{\prime }\left(
x\right) +\bar{q}\left( x\right) z\left( x\right) =0  \label{Eqn92}
\end{equation}%
is transformed by the variable change $y\left( x\right) =g\left( x\right)
z\left( x\right) $ to give%
\begin{equation}
y^{\prime \prime }\left( x\right) +p\left( x\right) y^{\prime }\left(
x\right) +q\left( x\right) y\left( x\right) =0  \label{Eqn93}
\end{equation}%
and then from equation (\ref{Eqn73}) the function $q\left( x\right) $ is
given by%
\begin{equation}
q\left( x\right) =\frac{1}{2}\left( p\left( x\right) -\bar{p}\left( x\right)
\right) ^{\prime }+\frac{p^{2}\left( x\right) -\bar{p}^{2}\left( x\right) }{4%
}+\bar{q}\left( x\right) \text{.}  \label{Eqn94}
\end{equation}%
If $\bar{p}\left( x\right) $ and $\bar{q}\left( x\right) $ are considered
known and $q\left( x\right) $ is to be specified, solving for $p\left(
x\right) $ as an unknown gives a Riccati equation [5]. Although nonlinear,
it is a well-known equation with many useful exact solutions [5]. This can
be illustrated with the associated Legendre equation, which has a relatively
cumbersome $q\left( x\right) ,$ in the context of applying equation (\ref%
{Eqn1}): 
\begin{equation}
z^{\prime \prime }\left( x\right) -\frac{2x}{1-x^{2}}z^{\prime }\left(
x\right) +\left( \frac{\nu \left( \nu +1\right) }{1-x^{2}}-\frac{\mu ^{2}}{%
\left( 1-x^{2}\right) ^{2}}\right) z\left( x\right) =0\text{.}  \label{Eqn95}
\end{equation}%
Under a general dependent variable change, the new $q\left( x\right) $ is
given by (\ref{Eqn94}) as%
\begin{equation}
\frac{1}{2}\frac{dp}{dx}+\frac{p^{2}}{4}+\frac{\nu \left( \nu +1\right) }{%
1-x^{2}}-\frac{\mu ^{2}-1}{\left( 1-x^{2}\right) ^{2}}  \label{Eqn96}
\end{equation}%
which has a simple solution%
\begin{equation}
p\left( x\right) =-\frac{2\left( x-\mu \right) }{1-x^{2}}  \label{Eqn97}
\end{equation}%
and the transformed associated Legendre equation is therefore%
\begin{equation}
y^{\prime \prime }\left( x\right) -\frac{2\left( x-\mu \right) }{1-x^{2}}%
y^{\prime }\left( x\right) +\frac{\nu \left( \nu +1\right) }{1-x^{2}}y\left(
x\right) =0\text{.}  \label{Eqn98}
\end{equation}%
For this differential equation $f\left( x\right) $ and $g\left( x\right) $
are given by equations (\ref{Eqn1a}) and (\ref{Eqn69}), respectively, as%
\begin{equation}
f\left( x\right) =\left( 1-x^{2}\right) \left( \frac{1+x}{1-x}\right) ^{\mu }
\label{Eqn99}
\end{equation}%
\begin{equation}
g\left( x\right) =\left( \frac{1-x}{1+x}\right) ^{\mu /2}  \label{Eqn100}
\end{equation}%
and hence the general solution of (\ref{Eqn98}) is%
\begin{equation}
y\left( x\right) =\left( \frac{1-x}{1+x}\right) ^{\mu /2}\left( C_{1}P_{\nu
}^{\mu }\left( x\right) +C_{2}Q_{\nu }^{\mu }\left( x\right) \right) \text{.}
\label{Eqn101}
\end{equation}%
Simple integrals can be obtained by taking $h\left( x\right) $ to be either
of the independent solutions of the equation%
\begin{equation}
h^{\prime \prime }\left( x\right) -\frac{2\left( x-\mu \right) }{1-x^{2}}%
h^{\prime }\left( x\right) =0  \label{Eqn102}
\end{equation}%
which are%
\begin{equation}
h\left( x\right) =1  \label{Eqn103}
\end{equation}%
\begin{equation}
h\left( x\right) =\left( \frac{1-x}{1+x}\right) ^{\mu }\text{.}
\label{Eqn104}
\end{equation}%
Substituting $h\left( x\right) =1$ into equation (\ref{Eqn1}) gives%
\begin{equation*}
\dint \left( \frac{1+x}{1-x}\right) ^{\mu /2}\QATOPD\{ \} {P_{\nu }^{\mu
}\left( x\right) }{Q_{\nu }^{\mu }\left( x\right) }dx=
\end{equation*}%
\begin{equation}
\frac{1}{\nu \left( \nu +1\right) }\left( \frac{1+x}{1-x}\right) ^{\mu
/2}\QATOPD\{ \} {\mu P_{\nu }^{\mu }\left( x\right) +\left( x^{2}-1\right)
P_{\nu }^{\prime \mu }\left( x\right) }{\mu Q_{\nu }^{\mu }\left( x\right)
+\left( x^{2}-1\right) Q_{\nu }^{\prime \mu }\left( x\right) }
\label{Eqn105}
\end{equation}%
which from the recurrence relation [1]:%
\begin{equation}
\left( 1-x^{2}\right) P_{\nu }^{\prime \mu }\left( x\right) =\nu xP_{\nu
}^{\mu }\left( x\right) -\left( \nu +\mu \right) P_{\nu -1}\left( x\right)
\label{Eqn106}
\end{equation}%
can be expressed as%
\begin{equation*}
\dint \left( \frac{1+x}{1-x}\right) ^{\mu /2}\QATOPD\{ \} {P_{\nu }^{\mu
}\left( x\right) }{Q_{\nu }^{\mu }\left( x\right) }dx=
\end{equation*}%
\begin{equation}
\frac{1}{\nu \left( \nu +1\right) }\left( \frac{1+x}{1-x}\right) ^{\mu
/2}\QATOPD\{ \} {\left( \nu x+\mu \right) P_{\nu }^{\mu }\left( x\right)
-\left( \nu +\mu \right) P_{\nu -1}^{\mu }\left( x\right) }{\left( \nu x+\mu
\right) Q_{\nu }^{\mu }\left( x\right) -\left( \nu +\mu \right) Q_{\nu
-1}^{\mu }\left( x\right) }\text{.}  \label{Eqn107}
\end{equation}%
Substituting equation (\ref{Eqn104}) into equation (\ref{Eqn1}) gives after
some reduction:%
\begin{equation*}
\dint \left( \frac{1-x}{1+x}\right) ^{\mu /2}\QATOPD\{ \} {P_{\nu }^{\mu
}\left( x\right) }{Q_{\nu }^{\mu }\left( x\right) }dx=
\end{equation*}%
\begin{equation}
\frac{1}{\nu \left( \nu +1\right) }\left( \frac{1-x}{1+x}\right) ^{\mu /2}%
\left[ \QATOPD\{ \} {\left( \nu x-\mu \right) P_{\nu }^{\mu }\left( x\right)
-\left( \nu +\mu \right) P_{\nu -1}\left( x\right) }{\left( \nu x-\mu
\right) Q_{\nu }^{\mu }\left( x\right) -\left( \nu +\mu \right) Q_{\nu
-1}\left( x\right) }\right] \text{.}  \label{Eqn109}
\end{equation}%
Equations (\ref{Eqn107}) and (\ref{Eqn109}) are previously known integrals
which are tabulated in a combined form in [4]. A different integral is
obtained by taking $h\left( x\right) $ to be the solution of 
\begin{equation}
-\frac{2\left( x-\mu \right) }{1-x^{2}}h^{\prime }\left( x\right) +\frac{\nu
\left( \nu +1\right) }{1-x^{2}}h\left( x\right) =0  \label{Eqn110}
\end{equation}%
which is: 
\begin{equation}
h\left( x\right) =\left( x-\mu \right) ^{\nu \left( \nu +1\right) /2}\text{.}
\label{Eqn111}
\end{equation}%
Substituting equation (\ref{Eqn111})\ into equation (\ref{Eqn1}) gives after
some reduction%
\begin{equation*}
\dint \left( x-\mu \right) ^{\nu \left( \nu +1\right) /2-2}\frac{\left(
1+x\right) ^{\mu /2+1}}{\left( 1-x\right) ^{\mu /2-1}}\QATOPD\{ \} {P_{\nu
}^{\mu }\left( x\right) }{Q_{\nu }^{\mu }\left( x\right) }dx=\left( x-\mu
\right) ^{\nu \left( \nu +1\right) /2}\left( \frac{1+x}{1-x}\right) ^{\mu
/2}\times
\end{equation*}%
\begin{equation*}
\left( \frac{2\left( 1-x^{2}\right) }{\left( \nu -1\right) \left( \nu
+2\right) \left( x-\mu \right) }\QATOPD\{ \} {P_{\nu }^{\mu }\left( x\right)
}{Q_{\nu }^{\mu }\left( x\right) }+\right.
\end{equation*}%
\begin{equation}
\left. \frac{4}{\left( \nu -1\right) \nu \left( \nu +1\right) \left( \nu
+2\right) }\QATOPD\{ \} {\left( \nu x+\mu \right) P_{\nu }^{\mu }\left(
x\right) -\left( \nu +\mu \right) P_{\nu -1}^{\mu }\left( x\right) }{\left(
\nu x+\mu \right) Q_{\nu }^{\mu }\left( x\right) -\left( \nu +\mu \right)
Q_{\nu -1}^{\mu }\left( x\right) }\right) \text{.}  \label{Eqn112}
\end{equation}%
This integral seems to be new.

\section{Examples for some special functions}

There are an unlimited number of cases as $h\left( x\right) $ is arbitrary.
The art of using equation (\ref{Eqn1}) is to choose $h\left( x\right) $ to
give an interesting integral.

\subsection{Bessel Functions}

Equation (\ref{Eqn1}) for Bessel's equation gives:%
\begin{equation*}
\dint x\left( h^{\prime \prime }\left( x\right) +\frac{1}{x}h^{\prime
}\left( x\right) +\left( 1-\frac{n^{2}}{x^{2}}\right) h\left( x\right)
\right) J_{n}\left( x\right) dx=
\end{equation*}%
\begin{equation}
x\left( h^{\prime }\left( x\right) J_{n}\left( x\right) -h\left( x\right)
J_{n}^{\prime }\left( x\right) \right)  \label{Eqn113}
\end{equation}%
and substituting $h\left( x\right) =x^{m}$ in this equation:%
\begin{equation}
\dint x^{m+1}\left( 1+\frac{m^{2}-n^{2}}{x^{2}}\right) J_{n}\left( x\right)
dx=x^{m+1}\left( \frac{mJ_{n}\left( x\right) }{x}-J_{n}^{\prime }\left(
x\right) \right)  \label{Eqn114}
\end{equation}%
which from the Bessel recurrence%
\begin{equation}
J_{n}^{\prime }\left( x\right) =\frac{1}{2}\left( J_{n-1}\left( x\right)
-J_{n+1}\left( x\right) \right)  \label{Eqn115}
\end{equation}%
can be expressed as:%
\begin{equation*}
\dint x^{m+1}\left( 1+\frac{m^{2}-n^{2}}{x^{2}}\right) J_{n}\left( x\right)
dx=
\end{equation*}%
\begin{equation}
x^{m+1}\left( \frac{mJ_{n}\left( x\right) }{x}+\frac{1}{2}\left(
J_{n+1}\left( x\right) -J_{n-1}\left( x\right) \right) \right) \text{.}
\label{Eqn116}
\end{equation}%
The second term in equation (\ref{Eqn116}) vanishes for $m=\pm n$, and the
Bessel recurrence%
\begin{equation}
\frac{mJ_{n}\left( x\right) }{x}=\frac{1}{2}\left( J_{n-1}\left( x\right)
+J_{n+1}\left( x\right) \right)  \label{Eqn117}
\end{equation}%
then gives the elementary cases%
\begin{equation}
\dint x^{n+1}J_{n}\left( x\right) dx=x^{n+1}J_{n+1}\left( x\right) dx
\label{Eqn118}
\end{equation}%
\begin{equation}
\dint x^{-n+1}J_{n}\left( x\right) dx=-x^{-n+1}J_{n-1}\left( x\right) dx%
\text{.}  \label{Eqn119}
\end{equation}%
The general integral given by equation (\ref{Eqn116}) appears to be new. For 
$n=0$ and $m=1$ this equation gives the simple expression%
\begin{equation}
\dint \left( 1+x^{2}\right) J_{0}\left( x\right) dx=xJ_{0}\left( x\right)
+x^{2}J_{1}\left( x\right)  \label{Eqn120}
\end{equation}%
which can be readily verified by differentiation.

Substituting $h\left( x\right) =x^{m}\QATOPD\{ \} {\sin x}{\cos x}$ in
equation (\ref{Eqn113}) gives the integral%
\begin{equation*}
\dint \left( \left( m^{2}-n^{2}\right) x^{m-1}\QATOPD\{ \} {\sin x}{\cos
x}+\left( 2m+1\right) x^{m}\QATOPD\{ \} {\cos x}{-\sin x}\right) J_{n}\left(
x\right) dx
\end{equation*}%
\begin{equation}
=x^{m+1}\left[ \left( \frac{mJ_{n}\left( x\right) }{x}-J_{n}^{\prime }\left(
x\right) \right) \QATOPD\{ \} {\sin x}{\cos x}+J_{n}\left( x\right)
\QATOPD\{ \} {\cos x}{-\sin x}\right] \text{.}  \label{Eqn121}
\end{equation}%
The integrand of (\ref{Eqn121}) reduces to a single term for $m=n$ or $%
m=-1/2 $. Employing the Bessel recurrence:%
\begin{equation}
J_{n}^{\prime }\left( x\right) =\frac{nJ_{n}\left( x\right) }{x}%
-J_{n+1}\left( x\right)  \label{Eqn122}
\end{equation}%
gives for $m=n$:%
\begin{equation}
\dint x^{n}\QATOPD\{ \} {\sin x}{\cos x}J_{n}\left( x\right) dx=\frac{x^{n+1}%
}{2n+1}\QATOPD\{ \} {J_{n}\left( x\right) \sin x-J_{n+1}\left( x\right) \cos
x}{J_{n}\left( x\right) \cos x+J_{n+1}\left( x\right) \sin x}  \label{Eqn123}
\end{equation}%
and for $m=-1/2:$%
\begin{equation*}
\dint x^{-3/2}J_{n}\left( x\right) \QATOPD\{ \} {\sin x}{\cos x}dx=
\end{equation*}%
\begin{equation}
=\left( \frac{J_{n}\left( x\right) }{\left( n-1/2\right) \sqrt{x}}-\frac{%
\sqrt{x}J_{n+1}\left( x\right) }{n^{2}-1/4}\right) \QATOPD\{ \} {\sin
x}{\cos x}-\frac{\sqrt{x}J_{n}\left( x\right) }{n^{2}-1/4}\QATOPD\{ \} {\cos
x}{-\sin x}\text{.}  \label{Eqn124}
\end{equation}%
Equation (\ref{Eqn123}) $\allowbreak $is tabulated in [3]. Equation (\ref%
{Eqn124}) is also tabulated in [3] but unfortunately with multiple
typographical errors. The same errors are also present in the original
Russian edition [11], but it is evident that the authors of [3,11] did
originally have equation (\ref{Eqn124}).

Substituting $h\left( x\right) =x^{n}\ln \left( x\right) $ in equation (\ref%
{Eqn113}) gives%
\begin{equation}
\dint \left( 2nx^{n-1}+x^{n+1}\ln \left( x\right) \right) J_{n}\left(
x\right) dx=x^{n}\left( 1+n\ln x\right) J_{n}\left( x\right) -x^{n+1}\ln
\left( x\right) J_{n}^{\prime }\left( x\right)  \label{Eqn125}
\end{equation}%
which can be simplified with equation (\ref{Eqn122}) to give%
\begin{equation}
\dint \left( 2nx^{n-1}+x^{n+1}\ln \left( x\right) \right) J_{n}\left(
x\right) dx=x^{n}J_{n}\left( x\right) +\ln \left( x\right)
x^{n+1}J_{n+1}\left( x\right) \text{.}  \label{Eqn126}
\end{equation}%
For $n=0$, equation (\ref{Eqn126}) reduces to%
\begin{equation}
\dint x\ln \left( x\right) J_{0}\left( x\right) dx=J_{0}\left( x\right)
+x\ln \left( x\right) J_{1}\left( x\right) \text{.}  \label{Eqn127}
\end{equation}%
$\allowbreak $Despite its simplicity, equation (\ref{Eqn127}) appears to be
new, and the Mathematica software [12] is unable to evaluate this integral.

Substituting the power function $h\left( x\right) =x^{n}$ into equation (\ref%
{Eqn1}) with $y\left( x\right) $ given by the Airy equation (\ref{Eqn51a})
gives the integral:%
\begin{equation}
\dint \left( n\left( n-1\right) x^{n-2}-x^{n+1}\right) \QATOPD\{ \} {\func{Ai%
}\left( x\right) }{\func{Bi}\left( x\right) }dx=nx^{n-1}\QATOPD\{ \} {\func{%
Ai}\left( x\right) }{\func{Bi}\left( x\right) }-x^{n}\QATOPD\{ \} {\func{Ai}%
^{\prime }\left( x\right) }{\func{Bi}^{\prime }\left( x\right) }\text{.}
\label{Eqn128}
\end{equation}%
Equation (\ref{Eqn128}) reduces to a single term for $n=0$ and $n=1$, to
give the well-known integrals:%
\begin{equation}
\dint x\QATOPD\{ \} {\func{Ai}\left( x\right) }{\func{Bi}\left( x\right)
}dx=\QATOPD\{ \} {\func{Ai}^{\prime }\left( x\right) }{\func{Bi}^{\prime
}\left( x\right) }  \label{Eqn128a}
\end{equation}%
\begin{equation}
\dint x^{2}\QATOPD\{ \} {\func{Ai}\left( x\right) }{\func{Bi}\left( x\right)
}dx=x\QATOPD\{ \} {\func{Ai}^{\prime }\left( x\right) }{\func{Bi}^{\prime
}\left( x\right) }-\QATOPD\{ \} {\func{Ai}\left( x\right) }{\func{Bi}\left(
x\right) }\text{.}  \label{Eqn128b}
\end{equation}%
Defining%
\begin{equation}
I_{n}\left( x\right) =\dint x^{n}\QATOPD\{ \} {\func{Ai}\left( x\right) }{%
\func{Bi}\left( x\right) }dx  \label{Eqn128c}
\end{equation}%
then equation (\ref{Eqn128}) can be expressed as a recursion relation for
the $I_{n}\left( x\right) $:%
\begin{equation}
I_{n+3}\left( x\right) =\left( n+1\right) \left( n+2\right) I_{n}\left(
x\right) -\left( n+2\right) x^{n+1}\QATOPD\{ \} {\func{Ai}\left( x\right) }{%
\func{Bi}\left( x\right) }+x^{n+2}\QATOPD\{ \} {\func{Ai}^{\prime }\left(
x\right) }{\func{Bi}^{\prime }\left( x\right) }\text{.}  \label{Eqn128d}
\end{equation}%
Using equations (\ref{Eqn51c})-(\ref{Eqn51d}) and (\ref{Eqn128a})-(\ref%
{Eqn128b}) as starting formulas, equation (\ref{Eqn128d}) allows closed form
solutions to be obtained by upward recursion for all $n\in 
\mathbb{N}
_{0}$. Downward recursion to negative integer values of $n$ is blocked when (%
\ref{Eqn128a}) and (\ref{Eqn128b}) are the starting formulas, but it can be
done with either (\ref{Eqn51c}) or (\ref{Eqn51d}) as a starting formula.

The Airy functions satisfy the differential equation (\ref{Eqn51a}) and from
this equation it follows that the slightly more general differential equation%
\begin{equation}
y^{\prime \prime }\left( x\right) -\left( x-\alpha \right) y\left( x\right)
=0  \label{Eqn129}
\end{equation}%
has the general solution%
\begin{equation}
y\left( x\right) =C_{1}\func{Ai}\left( x-\alpha \right) +C_{2}\func{Bi}%
\left( x-\alpha \right) \allowbreak \text{.}  \label{Eqn130}
\end{equation}%
Substituting $h\left( x\right) =\sin \left( x+\phi \right) $ into equation (%
\ref{Eqn1}) with $y\left( x\right) $ a solution of equation (\ref{Eqn129})
for $\alpha =1$ gives%
\begin{equation*}
\dint x\sin \left( x+\phi \right) \QATOPD\{ \} {\func{Ai}\left( x-1\right) }{%
\func{Bi}\left( x-1\right) }dx=
\end{equation*}%
\begin{equation}
\sin \left( x+\phi \right) \QATOPD\{ \} {\func{Ai}^{\prime }\left(
x-1\right) }{\func{Bi}^{\prime }\left( x-1\right) }-\cos \left( x+\phi
\right) \QATOPD\{ \} {\func{Ai}\left( x-1\right) }{\func{Bi}\left(
x-1\right) }\text{.}  \label{Eqn132}
\end{equation}%
Similarly, substituting $h\left( x\right) =e^{\pm x}$ into equation (\ref%
{Eqn1}) for $\alpha =-1$ gives%
\begin{equation}
\dint xe^{\pm x}\QATOPD\{ \} {\func{Ai}\left( x+1\right) }{\func{Bi}\left(
x+1\right) }dx=e^{\pm x}\left( \QATOPD\{ \} {\func{Ai}^{\prime }\left(
x+1\right) }{\func{Bi}^{\prime }\left( x+1\right) }\mp \QATOPD\{ \} {\func{Ai%
}\left( x+1\right) }{\func{Bi}\left( x+1\right) }\right) \text{.}
\label{Eqn133}
\end{equation}%
Substituting $h\left( x\right) =\func{Ai}\left( x-\beta \right) $ and $%
h\left( x\right) =\func{Ai}\left( x-\beta \right) $ into equation (\ref{Eqn1}%
) gives, respectively,%
\begin{equation*}
\dint \QATOPD\{ \} {\func{Ai}\left( x-\beta \right) \func{Ai}\left( x-\alpha
\right) }{\func{Ai}\left( x-\beta \right) \func{Bi}\left( x-\alpha \right)
}dx=
\end{equation*}%
\begin{equation}
\frac{1}{\alpha -\beta }\QATOPD\{ \} {\func{Ai}^{\prime }\left( x-\beta
\right) \func{Ai}\left( x-\alpha \right) -\func{Ai}\left( x-\beta \right) 
\func{Ai}^{\prime }\left( x-\alpha \right) }{\func{Ai}^{\prime }\left(
x-\beta \right) \func{Bi}\left( x-\alpha \right) -\func{Ai}\left( x-\beta
\right) \func{Bi}^{\prime }\left( x-\alpha \right) }  \label{Eqn134}
\end{equation}%
\begin{equation*}
\dint \QATOPD\{ \} {\func{Bi}\left( x-\beta \right) \func{Ai}\left( x-\alpha
\right) }{\func{Bi}\left( x-\beta \right) \func{Bi}\left( x-\alpha \right)
}dx=
\end{equation*}%
\begin{equation}
\frac{1}{\alpha -\beta }\QATOPD\{ \} {\func{Bi}^{\prime }\left( x-\beta
\right) \func{Ai}\left( x-\alpha \right) -\func{Ai}\left( x-\beta \right) 
\func{Ai}^{\prime }\left( x-\alpha \right) }{\func{Bi}^{\prime }\left(
x-\beta \right) \func{Bi}\left( x-\alpha \right) -\func{Bi}\left( x-\beta
\right) \func{Bi}^{\prime }\left( x-\alpha \right) }\text{.}  \label{Eqn135}
\end{equation}%
The integrals given by equations (\ref{Eqn132})-(\ref{Eqn135}) appear to be
new.

\subsection{Gauss hypergeometric functions}

The Gauss hypergeometric function $y\left( x\right) ={}_{2}F_{1}\left(
\alpha ,\beta ;\gamma ;x\right) $ satisfies the differential equation [1]:

\begin{equation}
y^{\prime \prime }\left( x\right) +\frac{\gamma -\left( \alpha +\beta
+1\right) x}{x\left( 1-x\right) }y^{\prime }\left( x\right) -\frac{\alpha
\beta }{x\left( 1-x\right) }y\left( x\right) =0  \label{Eqn136}
\end{equation}%
and for this equation $f\left( x\right) $ is given by equation (\ref{Eqn1a})
as 
\begin{equation}
f\left( x\right) =x^{\gamma }\left( 1-x\right) ^{\alpha +\beta -\gamma +1}%
\text{.}  \label{Eqn137}
\end{equation}

Substituting $h\left( x\right) =1$ into equation (\ref{Eqn1}) gives%
\begin{equation*}
\dint x^{\gamma -1}\left( 1-x\right) ^{\alpha +\beta -\gamma
}{}_{2}F_{1}\left( \alpha ,\beta ;\gamma ;x\right) =
\end{equation*}%
\begin{equation}
\frac{1}{\gamma }x^{\gamma }\left( 1-x\right) ^{\alpha +\beta -\gamma
+1}{}_{2}F_{1}\left( \alpha +1,\beta +1;\gamma +1;x\right)  \label{Eqn138}
\end{equation}%
which is a tabulated integral [4].

We can take $h\left( x\right) $ to be a solution of%
\begin{equation}
\frac{\gamma -\left( \alpha +\beta +1\right) x}{x\left( 1-x\right) }%
y^{\prime }\left( x\right) -\frac{\alpha \beta }{x\left( 1-x\right) }y\left(
x\right) =0  \label{Eqn139}
\end{equation}%
which gives%
\begin{equation}
h\left( x\right) =\left( \gamma -\left( \alpha +\beta +1\right) x\right) ^{-%
\frac{\alpha \beta }{\alpha +\beta +1}}  \label{Eqn140}
\end{equation}%
and substituting this $h\left( x\right) $ into equation (\ref{Eqn1}) gives 
\begin{equation*}
\dint \left( \gamma -\left( \alpha +\beta +1\right) x\right) ^{-\frac{\alpha
\beta }{\alpha +\beta +1}-2}{}_{2}F_{1}\left( \alpha ,\beta ;\gamma
;x\right) dx=
\end{equation*}%
\begin{equation*}
\frac{x^{\gamma }\left( 1-x\right) ^{\alpha +\beta +1-\gamma }}{\left(
\alpha \beta +\alpha +\beta +1\right) }\left( \gamma -\left( \alpha +\beta
+1\right) x\right) ^{-\frac{\alpha \beta }{\alpha +\beta +1}{}}\times
\end{equation*}%
\begin{equation}
\left( {}\frac{_{2}F_{1}\left( \alpha ,\beta ;\gamma ;x\right) }{\gamma
-\left( \alpha +\beta +1\right) x}-{}\frac{_{2}F_{1}\left( \alpha +1,\beta
+1;\gamma +1;x\right) }{\gamma }\right)  \label{Eqn141}
\end{equation}

\subsubsection{Conjugate hypergeometric functions}

In equation (\ref{Eqn1}) we can choose $y\left( x\right) $ to be the
hypergeometric function%
\begin{equation}
y\left( x\right) ={}_{2}F_{1}\left( \alpha ,\beta ;\gamma ;x\right)
\label{Eqn142}
\end{equation}%
and $h\left( x\right) $ to be a conjugate hypergeometic function: 
\begin{equation}
h\left( x\right) ={}_{2}F_{1}\left( \alpha +\delta ,\beta -\delta ;\gamma
;x\right)  \label{Eqn143}
\end{equation}%
where $h\left( x\right) $ satisfies the equation%
\begin{equation}
h^{\prime \prime }\left( x\right) +\frac{\gamma -\left( \alpha +\beta
+1\right) x}{x\left( 1-x\right) }h^{\prime }\left( x\right) -\left( \frac{%
\alpha \beta }{x\left( 1-x\right) }-\frac{\delta \left( \alpha -\beta
+\delta \right) }{x\left( 1-x\right) }\right) h\left( x\right) =0\text{.}
\label{Eqn144}
\end{equation}%
This gives the integral 
\begin{equation*}
\gamma \delta \left( \alpha -\beta +\delta \right) \dint x^{\gamma -1}\left(
1-x\right) ^{\alpha +\beta -\gamma }{}_{2}F_{1}\left( \alpha +\delta ,\beta
-\delta ;\gamma ;x\right) {}_{2}F_{1}\left( \alpha ,\beta ;\gamma ;x\right)
dx=
\end{equation*}%
\begin{equation*}
x^{\gamma }\left( 1-x\right) ^{\alpha +\beta -\gamma +1}\left( \alpha \beta
_{2}F_{1}\left( \alpha +\delta ,\beta -\delta ;\gamma ;x\right)
{}_{2}F_{1}\left( \alpha +1,\beta +1;\gamma +1;x\right) \right.
\end{equation*}%
\begin{equation}
\left. -\left( \alpha +\delta \right) \left( \beta -\delta \right)
{}_{2}F_{1}\left( \alpha +\delta +1,\beta -\delta +1;\gamma +1;x\right)
{}_{2}F_{1}\left( \alpha ,\beta ;\gamma ;x\right) \right) \text{.}
\label{Eqn145}
\end{equation}

\section{Complete elliptic integrals of the first and second kinds}

In this section the standard notation will be adopted for elliptic
integrals, with the modulus denoted by $k$ and the complementary modulus by $%
k^{\prime }\equiv \sqrt{1-k^{2}}$. The complete elliptic integral of the
first kind $\mathbf{K}\left( k\right) $ is a solution of the differential
equation [1]:%
\begin{equation}
y^{\prime \prime }\left( k\right) +\left( \frac{1}{k}-\frac{2k}{1-k^{2}}%
\right) y^{\prime }\left( k\right) -\frac{1}{1-k^{2}}y\left( k\right) =0
\label{Eqn146}
\end{equation}%
which has $f\left( k\right) =k\left( 1-k^{2}\right) $ and has the general
solution%
\begin{equation}
y\left( k\right) =C_{1}\mathbf{K}\left( k\right) +C_{2}\mathbf{K}^{\prime
}\left( k\right)  \label{Eqn147}
\end{equation}%
where $\mathbf{K}^{\prime }\left( k\right) \equiv \mathbf{K}\left( k^{\prime
}\right) $. The corresponding equation for $\mathbf{E}\left( k\right) $, the
complete elliptic integral of the second kind, is [1]:%
\begin{equation}
z^{\prime \prime }\left( k\right) +\frac{1}{k}z^{\prime }\left( k\right) +%
\frac{1}{1-k^{2}}z\left( k\right) =0  \label{Eqn148}
\end{equation}%
which has $f\left( k\right) =k$ and the general solution%
\begin{equation}
z\left( k\right) =C_{1}\mathbf{E}\left( k\right) +C_{2}\left( \mathbf{E}%
^{\prime }\left( k\right) -\mathbf{K}^{\prime }\left( k\right) \right) \text{%
.}  \label{Eqn149}
\end{equation}%
Here, only the solutions $\mathbf{K}\left( k\right) $ and $\mathbf{E}\left(
k\right) $ of these equations will be considered, though the other solutions
also lead to interesting integrals.

\subsection{Transformations of the Equations}

These equations can be transformed to give many different forms for $p\left(
k\right) $, and only a few will be considered here. The dependent variable
name $y\left( k\right) $ will be retained for transformations of equation (%
\ref{Eqn146}) and the variable name $z\left( k\right) $ will be used for
transformations of equation (\ref{Eqn148}). Equation (\ref{Eqn146}) can be
transformed to be conjugate with both the Bessel equation (\ref{Eqn54}) and
equation (\ref{Eqn148}), which gives:%
\begin{equation}
y^{\prime \prime }\left( k\right) +\frac{1}{x}y^{\prime }\left( k\right) +%
\frac{1}{\left( 1-k^{2}\right) ^{2}}y\left( k\right) =0  \label{Eqn150}
\end{equation}%
and this equation has the general solution%
\begin{equation}
y\left( k\right) =C_{1}k^{\prime }\mathbf{K}\left( k\right) +\allowbreak
C_{2}k^{\prime }\mathbf{K}^{\prime }\left( k\right) \text{.}  \label{Eqn151}
\end{equation}%
Transforming equation (\ref{Eqn146}) to be conjugate with the associated
Legendre equation (\ref{Eqn95}) gives:%
\begin{equation}
y^{\prime \prime }\left( k\right) -\frac{2k}{1-k^{2}}y^{\prime }\left(
k\right) +\allowbreak \frac{1}{4k^{2}}y\left( k\right) =0\text{.}
\label{Eqn152}
\end{equation}%
This equation has $f\left( k\right) =\left( 1-k^{2}\right) $ and the general
solution:%
\begin{equation}
y\left( k\right) =C_{1}\sqrt{k}\mathbf{K}\left( k\right) +\allowbreak C_{2}%
\sqrt{k}\mathbf{K}^{\prime }\left( k\right) \text{.}  \label{Eqn153}
\end{equation}%
Equation (\ref{Eqn148}) can be transformed to be conjugate with equation (%
\ref{Eqn146}), which gives 
\begin{equation}
z^{\prime \prime }\left( k\right) +\left( \frac{1}{k}-\frac{2k}{1-k^{2}}%
\right) z^{\prime }\left( k\right) -\frac{1}{\left( 1-k^{2}\right) ^{2}}%
z\left( k\right) =0  \label{Eqn154}
\end{equation}%
and this equation has the general solution%
\begin{equation}
z\left( k\right) =C_{1}\frac{\mathbf{E}\left( k\right) }{k^{\prime }}+C_{2}%
\frac{\mathbf{E}^{\prime }\left( k\right) -\mathbf{K}^{\prime }\left(
k\right) }{k^{\prime }}\text{.}  \label{Eqn155}
\end{equation}%
Equation (\ref{Eqn148}) can also be transformed to be conjugate with
equation (\ref{Eqn152}) and the associated Legendre equation (\ref{Eqn95}).
The transformed equation can be expressed in either of the forms:%
\begin{equation}
z^{\prime \prime }\left( k\right) -\frac{2k}{1-k^{2}}z^{\prime }\left(
k\right) +\left( \frac{1}{4k^{2}}+\frac{1}{1-k^{2}}-\frac{1}{\left(
1-k^{2}\right) ^{2}}\right) z\left( k\right) =0  \label{Eqn156}
\end{equation}%
\begin{equation}
z^{\prime \prime }\left( k\right) -\frac{2k}{1-k^{2}}z^{\prime }\left(
k\right) +\left( \frac{1}{4k^{2}}-\frac{k^{2}}{\left( 1-k^{2}\right) ^{2}}%
\right) z\left( k\right) =0  \label{Eqn156a}
\end{equation}%
and has the general solution%
\begin{equation}
z\left( k\right) =C_{1}\frac{\sqrt{k}\mathbf{E}\left( k\right) }{k^{\prime }}%
+\frac{\sqrt{k}\left( \mathbf{E}^{\prime }\left( k\right) -\mathbf{K}%
^{\prime }\left( k\right) \right) }{k^{\prime }}\text{.}  \label{Eqn157}
\end{equation}%
An unlimited number of similar transformations are possible, each yielding
distinct integrals, but only equations (\ref{Eqn146}) and (\ref{Eqn148})
will be considered in detail here. The technique used below is to specify $%
h\left( x\right) $ as a solution to some fragment of the differential
equation under consideration. For each equation, this gives a long list of
integrals, most new and some very surprising. Since an unlimited number of
transformations of the equations are possible, an extremely large number of
interesting integrals can be generated. The lists of integrals given below
for equations (\ref{Eqn146}) and (\ref{Eqn148}) are not complete, as some
fragmentary solutions were rejected as not particularly interesting.
However, the reader will be able to generate many additional integrals.

\subsection{Integrals from fragments of the equation for the complete
elliptic integral of the first kind}

In equation (\ref{Eqn1}) with $y\left( k\right) $ a solution of equation (%
\ref{Eqn146}) we can take $h\left( k\right) =1$ and employ the recurrence
relation:%
\begin{equation}
\frac{d\mathbf{K}\left( k\right) }{dk}=\frac{\mathbf{E}\left( k\right) }{%
kk^{\prime 2}}-\frac{\mathbf{K}\left( k\right) }{k}  \label{Eqn158}
\end{equation}%
to obtain the integral%
\begin{equation}
\dint k\mathbf{K}\left( k\right) dk=\mathbf{E}\left( k\right) -k^{\prime 2}%
\mathbf{K}\left( k\right) \text{.}  \label{Eqn159}
\end{equation}%
The six fragmentary equations:%
\begin{equation}
h^{\prime \prime }\left( k\right) +\left( \frac{1}{k}-\frac{2k}{1-k^{2}}%
\right) h^{\prime }\left( k\right) =0  \label{Eqn160}
\end{equation}%
\begin{equation}
h^{\prime \prime }\left( k\right) +\frac{1}{k}h^{\prime }\left( k\right) =0
\label{Eqn161}
\end{equation}%
\begin{equation}
h^{\prime \prime }\left( k\right) -\frac{2k}{1-k^{2}}h^{\prime }\left(
k\right) =0  \label{Eqn162}
\end{equation}%
$\allowbreak $%
\begin{equation}
\left( \frac{1}{k}-\frac{2k}{1-k^{2}}\right) h^{\prime }\left( x\right) -%
\frac{1}{1-k^{2}}h\left( x\right) =0  \label{Eqn163}
\end{equation}%
\begin{equation}
\frac{1}{k}h^{\prime }\left( x\right) -\frac{1}{1-k^{2}}h\left( x\right) =0
\label{Eqn164}
\end{equation}%
\begin{equation}
\frac{-2k}{1-k^{2}}h^{\prime }\left( x\right) -\frac{1}{1-k^{2}}h\left(
x\right) =0  \label{Eqn165}
\end{equation}%
have the (non constant) respective solutions:%
\begin{equation}
h\left( k\right) =\ln \left( k/k^{\prime }\right)  \label{Eqn166}
\end{equation}%
\begin{equation}
h\left( k\right) =\ln \left( k\right)  \label{Eqn167}
\end{equation}%
\begin{equation}
h\left( k\right) =\func{arctanh}\left( k\right)  \label{Eqn168}
\end{equation}%
\begin{equation}
h\left( k\right) =\frac{1}{\left( 3k^{2}-1\right) ^{1/6}}  \label{Eqn170}
\end{equation}%
\begin{equation}
h\left( k\right) =\frac{1}{k^{\prime }}  \label{Eqn171}
\end{equation}%
\begin{equation}
h\left( k\right) =\frac{1}{\sqrt{k}}  \label{Eqn172}
\end{equation}%
and these give the six respective integrals:%
\begin{equation}
\dint k\ln \left( \frac{k}{k^{\prime }}\right) \mathbf{K}\left( k\right)
dk=\ln \left( \frac{k}{k^{\prime }}\right) \left( \mathbf{E}\left( k\right)
-k^{\prime 2}\mathbf{K}\left( k\right) \right) -\mathbf{K}\left( k\right)
\label{Eqn173}
\end{equation}%
\begin{equation}
\dint k\left( 2+\ln \left( k\right) \right) \mathbf{K}\left( k\right) dk=\ln
\left( k\right) \mathbf{E}\left( k\right) -k^{\prime 2}\left( 1+\ln \left(
k\right) \right) \mathbf{K}\left( k\right)  \label{Eqn174}
\end{equation}%
\begin{equation}
\dint \left( 1-k\func{arctanh}\left( k\right) \right) \mathbf{K}\left(
k\right) dk=k\mathbf{K}\left( k\right) -\func{arctanh}\left( k\right) \left( 
\mathbf{E}\left( k\right) -k^{\prime ^{2}}\mathbf{K}\left( k\right) \right)
\label{Eqn175}
\end{equation}%
\begin{equation}
\dint \frac{k\left( 1-k^{2}\right) \left( 1+4k^{2}\right) }{\left(
3k^{2}-1\right) ^{13/6}}\mathbf{K}\left( k\right) dx=\frac{\left(
2k^{2}-1\right) \left( 1-k^{2}\right) }{\left( 3k^{2}-1\right) ^{7/6}}%
\mathbf{K}\left( k\right) -\frac{\mathbf{E}\left( k\right) }{\left(
3k^{2}-1\right) ^{1/6}}  \label{Eqn176}
\end{equation}%
\begin{equation}
\dint \frac{k\mathbf{K}\left( k\right) }{k^{\prime 3}}dk=\frac{\mathbf{K}%
\left( k\right) -\mathbf{E}\left( k\right) }{k^{\prime }}  \label{Eqn177}
\end{equation}%
\begin{equation}
\dint \frac{k^{\prime 2}}{k^{3/2}}\mathbf{K}\left( x\right) dx=\frac{1}{%
\sqrt{k}}\left( 2k^{\prime 2}\mathbf{K}\left( k\right) -4\mathbf{E}\left(
k\right) \right) \text{.}  \label{Eqn178}
\end{equation}%
Equations (\ref{Eqn159}) and (\ref{Eqn177}) are given in [4] $\allowbreak
\allowbreak \allowbreak $but equations (\ref{Eqn173})-(\ref{Eqn176}) and (%
\ref{Eqn178}) appear to be new. Additional fragmentary equations can be
constructed using the identity:%
\begin{equation}
\frac{1}{1-k^{2}}\equiv 1+\frac{k^{2}}{1-k^{2}}  \label{Eqn179}
\end{equation}%
which gives equation (\ref{Eqn146}) in the alternative form:%
\begin{equation}
y^{\prime \prime }\left( k\right) +\left( \frac{1}{k}-\frac{2k}{1-k^{2}}%
\right) y^{\prime }\left( k\right) -\left( 1+\frac{k^{2}}{1-k^{2}}\right)
y^{\prime }\left( k\right) =0\text{.}  \label{Eqn180}
\end{equation}%
Useful fragmentary equations from equation (\ref{Eqn180}) are:%
\begin{equation}
h^{\prime \prime }\left( k\right) +\frac{1}{k}h^{\prime }\left( k\right) -h=0
\label{Eqn181}
\end{equation}%
\begin{equation}
h^{\prime \prime }\left( k\right) -h=0  \label{Eqn182}
\end{equation}%
\begin{equation}
\left( \frac{1}{k}-\frac{2k}{1-k^{2}}\right) h^{\prime }\left( k\right) -h=0
\label{Eqn183}
\end{equation}%
\begin{equation}
-\frac{2k}{1-k^{2}}h^{\prime }\left( k\right) -h=0  \label{Eqn184}
\end{equation}%
\begin{equation}
\frac{1}{k}h^{\prime }\left( k\right) -h=0  \label{Eqn185}
\end{equation}%
\begin{equation}
\left( \frac{1}{k}-\frac{2k}{1-k^{2}}\right) h^{\prime }\left( k\right) -%
\frac{k^{2}}{1-k^{2}}h=0  \label{Eqn186}
\end{equation}%
\begin{equation}
-\frac{2k}{1-k^{2}}h^{\prime }\left( k\right) -\frac{k^{2}}{1-k^{2}}h=0
\label{Eqn187}
\end{equation}%
\begin{equation}
\frac{1}{k}h^{\prime }\left( k\right) -\frac{k^{2}}{1-k^{2}}h=0\text{.}
\label{Eqn188}
\end{equation}%
These equations have the respective solutions%
\begin{equation}
h\left( k\right) =C_{1}K_{0}\left( k\right) +C_{2}I_{0}\left( k\right)
\allowbreak \allowbreak  \label{Eqn189}
\end{equation}%
\begin{equation}
h\left( k\right) =e^{\pm k}  \label{Eqn190}
\end{equation}%
\begin{equation}
h\left( k\right) =\frac{e^{\frac{1}{6}k^{2}}}{\left( 3k^{2}-1\right) ^{\frac{%
1}{9}}}  \label{Eqn191}
\end{equation}%
\begin{equation}
h\left( k\right) =\frac{C_{1}}{\sqrt{k}}e^{\frac{1}{4}k^{2}}  \label{Eqn192}
\end{equation}%
\begin{equation}
h\left( k\right) =C_{1}e^{\frac{1}{2}k^{2}}  \label{Eqn193}
\end{equation}%
\begin{equation}
h\left( k\right) =\frac{e^{-\frac{1}{6}k^{2}}}{\left( 3k^{2}-1\right) ^{%
\frac{1}{18}}}  \label{Eqn194}
\end{equation}%
\begin{equation}
h\left( k\right) =e^{-\frac{1}{4}k^{2}}  \label{Eqn195}
\end{equation}%
\begin{equation}
\frac{1}{k^{\prime }}e^{-\frac{1}{2}k^{2}}  \label{Eqn196}
\end{equation}%
which in turn yield the respective integrals: $\allowbreak \allowbreak
\allowbreak \allowbreak $%
\begin{equation*}
\dint k^{2}\QATOPD\{ \} {2K_{1}\left( k\right) -kK_{0}\left( k\right)
}{-2I_{1}\left( k\right) -kI_{0}\left( k\right) }\mathbf{K}\left( k\right)
dk=
\end{equation*}%
\begin{equation}
\QATOPD\{ \} {k^{\prime 2}\left( K_{0}\left( k\right) -kK_{1}\left( k\right)
\right) }{k^{\prime 2}\left( I_{0}\left( k\right) +kI_{1}\left( k\right)
\right) }\mathbf{K}\left( k\right) -\QATOPD\{ \} {K_{0}\left( k\right)
}{I_{0}\left( k\right) }\mathbf{E}\left( k\right)  \label{Eqn197}
\end{equation}%
\begin{equation}
\dint \left( 1-3k^{2}\mp k^{3}\right) e^{\pm k}\mathbf{K}\left( k\right)
dk=\left( \left( k\pm 1\right) k^{\prime 2}\mathbf{K}\left( k\right) \mp 
\mathbf{E}\left( k\right) \right) e^{\pm k}  \label{Eqn198}
\end{equation}%
\begin{equation*}
\dint \frac{k\left( 1-k^{2}+6k^{4}-9k^{6}-k^{8}\right) }{\left(
3k^{2}-1\right) ^{\frac{19}{9}}}e^{\frac{1}{6}k^{2}}\mathbf{K}\left(
k\right) dk=
\end{equation*}%
\begin{equation}
\left( \frac{\left( 1-k^{2}\right) \left( k^{4}+2k^{2}-1\right) }{3k^{2}-1}%
\mathbf{K}\left( k\right) -\mathbf{E}\left( k\right) \right) \frac{e^{\frac{1%
}{6}k^{2}}}{\left( 3k^{2}-1\right) ^{\frac{1}{9}}}  \label{Eqn199}
\end{equation}%
\begin{equation}
\dint \left( 1+k^{2}-5k^{4}-k^{6}\right) \frac{e^{\frac{1}{4}k^{2}}}{k^{%
\frac{3}{2}}}\mathbf{K}\left( k\right) dk=\left( 2\left( 1-k^{4}\right) 
\mathbf{K}\left( k\right) -4\mathbf{E}\left( k\right) \right) \frac{e^{\frac{%
1}{4}k^{2}}}{k^{1/2}}  \label{Eqn200}
\end{equation}%
\begin{equation}
\dint k\left( k^{4}+3k^{2}-1\right) e^{\frac{1}{2}k^{2}}\mathbf{K}\left(
k\right) dk=e^{\frac{1}{2}k^{2}}\left( \mathbf{E}\left( k\right) -\left(
1-k^{4}\right) \mathbf{K}\left( k\right) \right)  \label{Eqn201}
\end{equation}%
\begin{equation*}
\dint \frac{k\left( 1-k^{2}\right) \left( 1-9k^{2}+12k^{4}-k^{6}\right) }{%
\left( 3k^{2}-1\right) ^{\frac{37}{18}}}e^{-\frac{1}{6}k^{2}}\mathbf{K}%
\left( k\right) dk=
\end{equation*}%
\begin{equation}
\left( \frac{\left( 1-k^{2}\right) \left( 1-3k^{2}+k^{4}\right) }{3k^{2}-1}%
\mathbf{K}\left( k\right) +\mathbf{E}\left( k\right) \right) \frac{e^{-\frac{%
1}{6}k^{2}}}{\left( 3k^{2}-1\right) ^{\frac{1}{18}}}  \label{Eqn202}
\end{equation}%
\begin{equation}
\dint kk^{\prime 2}\left( \frac{k^{2}}{4}-2\right) e^{-\frac{1}{4}k^{2}}%
\mathbf{K}\left( k\right) dk=k^{\prime 2}\left[ \left( 1-\frac{k^{2}}{2}%
\right) \mathbf{K}\left( k\right) -\mathbf{E}\left( k\right) \right]
\label{Eqn203}
\end{equation}%
\begin{equation*}
\dint \frac{k\left( 5k^{2}-4k^{4}+k^{6}-1\right) }{k^{\prime 3}}e^{-\frac{1}{%
2}k^{2}}\mathbf{K}\left( k\right) dk=
\end{equation*}%
\begin{equation}
\frac{e^{-\frac{1}{2}k^{2}}}{k^{\prime }}\left( \left( k^{4}-k^{2}+1\right) 
\mathbf{K}\left( k\right) -\mathbf{E}\left( k\right) \right) \text{.}
\label{Eqn204}
\end{equation}

\subsection{Integrals from fragments of the equation for the complete
elliptic integral of the second kind}

In equation (\ref{Eqn1}) with $y\left( k\right) $ a solution of equation (%
\ref{Eqn148}) we can take $h\left( k\right) =1$ and employ the recurrence
relation:%
\begin{equation}
\frac{d\mathbf{E}\left( k\right) }{dk}=\frac{\mathbf{E}\left( k\right) -%
\mathbf{K}\left( k\right) }{k}  \label{Eqn205}
\end{equation}%
to obtain the integral%
\begin{equation}
\dint \frac{k\mathbf{E}\left( k\right) }{k^{\prime 2}}dk=\mathbf{K}\left(
k\right) -\mathbf{E}\left( k\right) \text{.}  \label{Eqn206}
\end{equation}%
Using the identity (\ref{Eqn179}) equation (\ref{Eqn148}) can be expressed
in the alternative form:%
\begin{equation}
z^{\prime \prime }\left( k\right) +\frac{1}{k}z^{\prime }\left( k\right)
+\left( 1+\frac{k^{2}}{1-k^{2}}\right) z\left( k\right) =0  \label{Eqn207}
\end{equation}%
and suitable fragmentary equations from equations (\ref{Eqn148}) and (\ref%
{Eqn207})\ are:%
\begin{equation}
h^{\prime \prime }\left( k\right) +\frac{1}{k}h^{\prime }\left( k\right) =0
\label{Eqn208}
\end{equation}%
\begin{equation}
h^{\prime \prime }\left( k\right) +\frac{1}{1-k^{2}}h\left( k\right) =0
\label{Eqn209}
\end{equation}%
\begin{equation}
h^{\prime \prime }\left( k\right) +\frac{1}{k}h^{\prime }\left( k\right) +h=0
\label{Eqn210}
\end{equation}%
\begin{equation}
h^{\prime \prime }\left( k\right) +h=0  \label{Eqn211}
\end{equation}%
$\allowbreak $%
\begin{equation}
\frac{1}{k}h^{\prime }\left( k\right) +\frac{1}{1-k^{2}}h\left( k\right) =0
\label{Eqn212}
\end{equation}%
\begin{equation}
\frac{1}{k}h^{\prime }\left( k\right) +\frac{k^{2}}{1-k^{2}}h\left( k\right)
=0  \label{Eqn213}
\end{equation}%
\begin{equation}
\frac{1}{k}h^{\prime }\left( k\right) +h=0\text{.}  \label{Eqn215}
\end{equation}%
Equations (\ref{Eqn208})-(\ref{Eqn215}) have the respective (non constant)
solutions%
\begin{equation}
h\left( k\right) =\ln \left( k\right)  \label{Eqn216}
\end{equation}%
\begin{equation}
h\left( k\right) =k^{\prime }\QATOPD\{ \} {P_{\varphi -1}^{1}\left( k\right)
}{Q_{\varphi -1}\left( k\right) }  \label{Eqn217}
\end{equation}%
\begin{equation}
h\left( k\right) =\QATOPD\{ \} {J_{0}\left( k\right) }{Y_{0}\left( k\right) }
\label{Eqn218}
\end{equation}%
\begin{equation}
h\left( k\right) =\QATOPD\{ \} {\sin \left( k\right) }{\cos \left( k\right) }
\label{Eqn219}
\end{equation}%
\begin{equation}
h\left( k\right) =k^{\prime }  \label{Eqn220}
\end{equation}%
\begin{equation}
h\left( k\right) =k^{\prime }e^{\frac{1}{2}k^{2}}  \label{Eqn221}
\end{equation}%
\begin{equation}
h\left( k\right) =e^{-\frac{1}{2}k^{2}}  \label{Eqn222}
\end{equation}%
where in equation (\ref{Eqn217}), $\varphi \equiv \left( 1+\sqrt{5}\right)
/2 $ is the Golden Ratio. The solutions (\ref{Eqn216})-(\ref{Eqn222}) give
the respective integrals: 
\begin{equation}
\dint \frac{k\ln \left( k\right) \mathbf{E}\left( k\right) }{k^{\prime 2}}%
dk=\left( 1-\ln \left( k\right) \right) \mathbf{E}\left( k\right) +\ln
\left( k\right) \mathbf{K}\left( k\right)  \label{Eqn223}
\end{equation}%
\begin{equation*}
\dint kk^{\prime }\QATOPD\{ \} {P_{\varphi -1}^{1}\left( k\right)
}{Q_{\varphi -1}^{1}\left( k\right) }\mathbf{E}\left( k\right) dk=
\end{equation*}%
\begin{equation}
\left( k^{\prime }\mathbf{K}\left( k\right) +\frac{\left( \varphi
k^{2}-1\right) \mathbf{E}\left( k\right) }{k^{\prime }}\right) \QATOPD\{ \}
{P_{\varphi -1}^{1}\left( k\right) }{Q_{\varphi -1}^{1}\left( k\right) }-%
\frac{\left( \varphi -1\right) k}{k^{\prime }}\mathbf{E}\left( k\right)
\QATOPD\{ \} {P_{\varphi }^{1}\left( k\right) }{Q_{\varphi }^{1}\left(
k\right) }  \label{Eqn225}
\end{equation}%
\begin{equation}
\dint \frac{k^{3}}{k^{\prime 2}}J_{0}\left( k\right) \mathbf{E}\left(
k\right) dk=J_{0}\left( k\right) \left( \mathbf{K}\left( k\right) -\mathbf{E}%
\left( k\right) \right) -kJ_{1}\left( k\right) \mathbf{E}\left( k\right)
\label{Eqn226}
\end{equation}%
\begin{equation*}
\dint \left( \QATOPD\{ \} {\cos \left( k\right) }{-\sin \left( k\right) }+%
\frac{k^{3}}{k^{\prime 2}}\QATOPD\{ \} {\sin \left( k\right) }{\cos \left(
k\right) }\right) \mathbf{E}\left( k\right) dk=
\end{equation*}%
\begin{equation}
k\QATOPD\{ \} {\cos \left( k\right) }{-\sin \left( k\right) }\mathbf{E}%
\left( k\right) -\QATOPD\{ \} {\sin \left( k\right) }{\cos \left( k\right)
}\left( \mathbf{E}\left( k\right) -\mathbf{K}\left( k\right) \right)
\label{Eqn227}
\end{equation}%
\begin{equation}
\dint \frac{k}{k^{\prime 3}}\mathbf{E}\left( k\right) dk=\frac{\mathbf{E}%
\left( k\right) }{k^{\prime }}-k^{\prime }\mathbf{K}\left( k\right)
\label{Eqn228}
\end{equation}%
\begin{equation*}
\dint kk^{\prime }\left( 1+\frac{k^{2}\left( k^{4}+k^{2}-3\right) }{%
k^{\prime 4}}\right) e^{\frac{1}{2}k^{2}}\mathbf{E}\left( k\right) dk=
\end{equation*}%
\begin{equation}
e^{\frac{1}{2}k^{2}}\left( k^{\prime }\mathbf{K}\left( k\right) -\left(
k^{\prime }+\frac{k^{4}}{k^{\prime }}\right) \mathbf{E}\left( k\right)
\right)  \label{Eqn229}
\end{equation}%
\begin{equation}
\dint \left( k^{\prime 2}-\left( \frac{k}{k^{\prime }}\right) ^{2}\right)
e^{-\frac{1}{2}k^{2}}\mathbf{E}\left( k\right) dk=\left[ \mathbf{K}\left(
k\right) -\left( 1+k^{2}\right) \mathbf{E}\left( k\right) \right] e^{-\frac{1%
}{2}k^{2}}\text{.}  \label{Eqn230}
\end{equation}

\subsection{Integrals involving the Golden Ratio}

Each transformed differential equation such as (\ref{Eqn150}),(\ref{Eqn152}%
),(\ref{Eqn154}) and (\ref{Eqn156}) yields its own set of fragmentary
equations and their corresponding integrals, similar in number to those
givenabove for equations (\ref{Eqn146}) and (\ref{Eqn148}), and too numerous
to present here. However, of particular interest are integrals related to
equation (\ref{Eqn225}) above, which link Legendre/associated Legendre
functions with $\mathbf{E}\left( k\right) $, $\mathbf{K}\left( k\right) $
and the Golden Ratio $\varphi $. Equation (\ref{Eqn156}) above has the
associated fragmentary equations:%
\begin{equation}
h^{\prime \prime }\left( k\right) -\frac{2k}{1-k^{2}}+\left( \frac{1}{1-k^{2}%
}-\frac{1}{\left( 1-k^{2}\right) ^{2}}\right) h\left( k\right) =0
\label{Eqn231}
\end{equation}%
\begin{equation}
h^{\prime \prime }\left( k\right) -\frac{2k}{1-k^{2}}h^{\prime }\left(
k\right) +\frac{1}{1-k^{2}}h\left( k\right) =0  \label{Eqn232}
\end{equation}%
which have the respective solutions:%
\begin{equation}
h\left( k\right) =\QATOPD\{ \} {P_{\varphi -1}^{1}\left( k\right)
}{Q_{\varphi -1}^{1}\left( k\right) }  \label{Eqn233}
\end{equation}%
\begin{equation}
h\left( k\right) =\QATOPD\{ \} {P_{\varphi -1}\left( k\right) }{Q_{\varphi
-1}\left( k\right) }  \label{Eqn234}
\end{equation}%
and substituting these results in equation (\ref{Eqn1}) gives the respective
integrals$:$%
\begin{equation*}
\dint \frac{k^{\prime }}{k^{3/2}}\QATOPD\{ \} {P_{\varphi -1}^{1}\left(
k\right) }{Q_{\varphi -1}^{1}\left( k\right) }\mathbf{E}\left( k\right) dk=%
\frac{4\sqrt{k}}{k^{\prime }}\times
\end{equation*}%
\begin{equation}
\left[ \left( \left[ \varphi k-\frac{3-k^{2}}{2k}\right] \mathbf{E}\left(
k\right) +\frac{k^{\prime }}{\sqrt{k}}\mathbf{K}\left( k\right) \right)
\QATOPD\{ \} {P_{\varphi -1}^{1}\left( k\right) }{Q_{\varphi -1}^{1}\left(
k\right) }-\left( \varphi -1\right) \mathbf{E}\left( k\right) \QATOPD\{ \}
{P_{\varphi }^{1}\left( k\right) }{Q_{\varphi }^{1}\left( k\right) }\right]
\label{Eqn235}
\end{equation}%
\begin{equation*}
\dint \frac{1-6k^{2}+k^{4}}{\left( kk^{\prime }\right) ^{3/2}}\QATOPD\{ \}
{P_{\varphi -1}\left( k\right) }{Q_{\varphi -1}\left( k\right) }\mathbf{E}%
\left( k\right) dk=\frac{4\sqrt{k}}{k^{\prime }}\times
\end{equation*}%
\begin{equation}
\left[ \left[ \left( \phi k-\frac{3-k^{2}}{2k}\right) \mathbf{E}\left(
k\right) +\frac{k^{\prime 2}}{k}\mathbf{K}\left( k\right) \right] \QATOPD\{
\} {P_{\varphi -1}\left( k\right) }{Q_{\varphi -1}\left( k\right) }-\varphi 
\mathbf{E}\left( k\right) \QATOPD\{ \} {P_{\varphi }\left( k\right)
}{Q_{\varphi }\left( k\right) }\right]  \label{Eqn236}
\end{equation}%
$\allowbreak $

\section{Comments and conclusions}

A new Lagrangian method has been presented for deriving indefinite integrals
of any function which obeys a second order linear ordinary differential
equation. The main result can also be proved very simply without variational
calculus. Various approaches have been presented to exploit the main formula
and transformation methods have been given to multiply the number of
interesting integrals obtainable. Sample results have been presented for
some special functions, but these only scratch the surface of what is
possible. The total number of indefinite integrals given in current tables
and handbooks can be multiplied considerably with the method. All of the
results presented here have been numerically checked using Mathematica [12].

\end{document}